\documentclass{amsart}

\usepackage{amsmath,amssymb}
\usepackage{enumitem}
\usepackage{array}
\usepackage{tikz}

\newtheorem{theorem}{Theorem}[section]
\newtheorem{lemma}[theorem]{Lemma}
\newtheorem{proposition}[theorem]{Proposition}
\newtheorem{corollary}[theorem]{Corollary}
\newtheorem{conjecture}[theorem]{Conjecture}

\theoremstyle{definition}
\newtheorem{definition}[theorem]{Definition}

\theoremstyle{remark}

\numberwithin{equation}{section}

\newcommand{\newword}[1]{\emph{\textbf{#1}}}

\newcommand{\comp}[1]{\mathbf{#1}}
\newcommand{\wt}{\mathbf{wt}}
\newcommand{\D}{\mathbb{D}}

\newcommand{\KD}{\mathrm{KD}}
\newcommand{\SSYT}{\mathrm{SSYT}}

\newcommand{\kohnert}{\mathfrak{K}}
\newcommand{\key}{\kappa}

\newcommand{\GL}{\mathrm{GL}}
\newcommand{\Schur}{\mathcal{S}}
\newcommand{\fSchur}{\Schur^{\mathrm{flag}}} 
\newcommand{\ch}{\mathrm{char}} 

\newcommand{\B}{\mathcal{B}}
\newcommand{\Dem}{\mathfrak{D}}

\newcommand{\Rect}{\operatorname{rect}}
\newcommand{\Label}{\mathcal{L}}

\newlength\cellsize 

\savebox2{
\begin{picture}(8,8)
\put(0,0){\line(1,0){8}}
\put(0,0){\line(0,1){8}}
\put(8,0){\line(0,1){8}}
\put(0,8){\line(1,0){8}}
\end{picture}}

\newcommand\boxify[1]{\def\thearg{#1}\def\nothing{}%
\ifx\thearg\nothing\vrule width0pt height\cellsize depth0pt%
  \else\hbox to 0pt{\usebox2\hss}\fi%
  \vbox to \cellsize{\vss\hbox to \cellsize{\hss$_{#1}$\hss}\vss}}

\savebox3{
\begin{picture}(8,8)
\put(4,4){\circle{8}}
\end{picture}}

\newcommand{\circify}[1]{\def\thearg{#1}\def\nothing{}%
\ifx\thearg\nothing\vrule width0pt height\cellsize depth0pt%
  \else\hbox to 0pt{\usebox3\hss}\fi%
  \vbox to \cellsize{\vss\hbox to \cellsize{\hss$_{#1}$\hss}\vss}}

\newcommand\nullify[1]{\def\thearg{#1}\def\nothing{}%
\ifx\thearg\nothing\vrule width0pt height\cellsize depth0pt%
  \else\hbox to 0pt{\hss}\fi%
  \vbox to \cellsize{\vss\hbox to \cellsize{\hss$_{#1}$\hss}\vss}}

\newcommand\tableau[1]{\vtop{\let\\=\cr
\setlength\baselineskip{-8000pt}
\setlength\lineskiplimit{8000pt}
\setlength\lineskip{0pt}
\halign{&\boxify{##}\cr#1\crcr}}}

\newcommand\cirtab[1]{\vline\vtop{\let\\=\cr
\setlength\baselineskip{-8000pt}
\setlength\lineskiplimit{8000pt}
\setlength\lineskip{0pt}
\halign{&\circify{##}\cr#1\crcr}}}

\newcommand\nulltab[1]{\vtop{\let\\=\cr
\setlength\baselineskip{-8000pt}
\setlength\lineskiplimit{8000pt}
\setlength\lineskip{0pt}
\halign{&\nullify{##}\cr#1\crcr}}}

\usepackage{xcolor}
\usepackage[colorinlistoftodos]{todonotes}

\begin{document}


\title[Kohnert's rule for flagged Schur modules]{Kohnert's rule for flagged Schur modules}

\author[Armon]{Sam Armon}
\email{armon@usc.edu}

\author[Assaf]{Sami Assaf}
\address{Department of Mathematics, University of Southern California, 3620 S. Vermont Ave., Los Angeles, CA 90089-2532, U.S.A.}
\email{shassaf@usc.edu}
\thanks{Work supported in part by NSF DMS-1763336.}

\author[Bowling]{Grant Bowling}
\email{gbowling@usc.edu}

\author[Ehrhard]{Henry Ehrhard}
\email{hehrhard@usc.edu}




\begin{abstract}
  Flagged Schur modules generalize the irreducible representations of the general linear group under the action of the Borel subalgebra. Their characters include many important generalizations of Schur polynomials, such as Demazure characters, flagged skew Schur polynomials, and Schubert polynomials. In this paper, we prove the characters of flagged Schur modules can be computed using a simple combinatorial algorithm due to Kohnert if and only if the indexing diagram is northwest. This gives a new proof that characters of flagged Schur modules are nonnegative sums of Demazure characters and gives a representation theoretic interpretation for Kohnert polynomials.
\end{abstract}

\maketitle

%
\section{Introduction}
%
\label{sec:introduction}

The irreducible representations of the general linear group were constructed explicitly by Weyl using the Ferrers diagrams of partitions. Similarly, for any \emph{diagram} $D$, a finite collection of cells in $\mathbb{Z}_+ \times \mathbb{Z}_+$, one can construct the \emph{Schur module} $\Schur_{D}$ for the general linear group. These modules have been widely studied, and their characters include generalizations of the Schur polynomials such as skew Schur polynomials and Stanley symmetric polynomials. More generally, one can construct the \emph{flagged Schur module} $\fSchur_{D}$, which carries the action of the Borel subgroup of lower triangular matrices. These modules are also well-studied for certain families of diagrams, and their characters include type A Demazure characters \cite{Dem74a}, flagged skew Schur polynomials \cite{RS95}, and Schubert polynomials \cite{LS82}.

The Borel-Weyl theorem also realizes irreducible representations for the general linear group as spaces of sections of line bundles of the flag manifold. In a similar way, this classical construction can be generalized to configuration varieties $\mathfrak{F}_D$ of families of subspaces of a fixed vector space with incidence relations specified by a diagram $D$. When the diagram $D$ has the \emph{northwest} property, Magyar \cite{Mag98} gave an explicit resolution proving $\mathfrak{F}_D$ is normal with rational singularities. Since the corresponding line bundle has no higher cohomology, its space of sections recovers the Schur module $\fSchur_D$. Magyar \cite{Mag98-2} uses the geometry to give a recurrence for the character in terms of degree-preserving divided difference operators that arise in Demazure's character formula \cite{Dem74a}. Reiner and Shimozono \cite{RS98} use this recurrence to show the characters of the corresponding flagged Schur modules expand nonnegatively into Demazure characters. 

In this paper, we prove the characters of flagged Schur modules for northwest diagrams can be computed  using Kohnert's rule \cite{Koh91}, thus resolving a conjecture of Assaf and Searles \cite{AS19}. They defined \emph{Kohnert polynomials} $\kohnert_D$ for diagrams $D$ using Kohnert's rule \cite{Koh91}, giving a common generalization of Demazure characters \cite{Koh91} and Schubert polynomials \cite{Ass-KR}. We prove the Kohnert polynomial $\kohnert_D$ is the character of the flagged Schur module $\fSchur_D$ for $D$ a northwest diagram by using the Demazure crystal structure for Kohnert polynomials \cite{Ass-KC} to show they also satisfy Magyar's recurrence. 

Our paper is structured as follows. In Section~\ref{sec:modules}, we review flagged Schur modules and Kohnert polynomials, along with the associated combinatorics. We state Magyar's recurrence for characters of flagged Schur modules, and begin to show it holds for Kohnert polynomials as well. In Section~\ref{sec:crystals}, we review crystals and reinterpret the last term in the desired recurrence for Kohnert polynomials as a statement about Demazure operators on crystals. In Section~\ref{sec:labeling}, we delve into the combinatorics of Kohnert diagrams to prove the desired result for Demazure operators on Kohnert crystals. Thus we prove Kohnert polynomials satisfy the same recurrence, and so coincide with characters of flagged Schur modules.

Our results have several corollaries. Using recent work of Assaf \cite{Ass-KC}, we give a new proof that characters of flagged Schur modules decompose as a nonnegative sum of Demazure characters. Our results also give an explicit representation theoretic interpretation for Kohnert polynomials indexed by northwest diagrams. Since Schubert polynomials are known to be characters of flagged Schur modules indexed by northwest shapes \cite{KP87,KP04}, this also gives a new proof of Kohnert's rule for Schubert polynomials \cite{Ass-KR}.

Magyar considered the more general class of $\%$-avoiding diagrams, which includes northwest diagrams. We show our result is tight in the sense that for $\%$-avoiding diagrams $D$ that are not northwest, the Kohnert polynomial does not agree with characters of the flagged Schur modules.

%
\section{Modules and polynomials}
%
\label{sec:modules}

In this section, we define the basic objects of study for this paper: flagged Schur modules, Demazure characters, and Kohnert polynomials.

\subsection{Schur modules for diagrams}
\label{sec:modules-diagram}

Consider the infinite, doubly indexed family of indeterminates $\{z_{i,j}\}_{i,j=1}^{\infty}$. A matrix $A \in \GL_N$ act on $\mathbb{C}[z_{i,j}]$ in the usual way,
  \[A \cdot z_{k,l} =
  \begin{cases}
    \sum_{i=1}^{N} A_{i,k} z_{i,l} & k \leq N \\
    z_{k,l} & k>N
  \end{cases}. \]

A \newword{diagram} $D$ is a finite collection of \emph{cells} in $\mathbb{Z}_+ \times \mathbb{Z}_+$. We use matrix convention, so that the northwest corner has index $(1,1)$. For example, Fig.~\ref{fig:diagrams} shows the Ferrers diagram for a partition, the skew diagram for nested partitions, the key diagram for a weak composition, the Rothe diagram for a permutation, and a generic diagram.

\begin{figure}[ht]
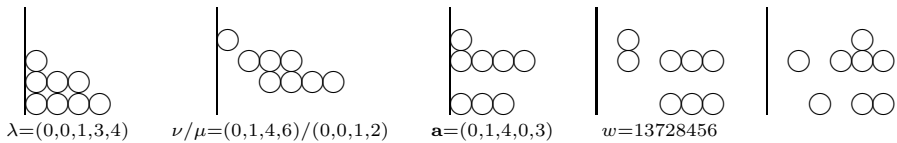

    \begin{displaymath}
    \arraycolsep=\cellsize
    \begin{array}{ccccc}
    \cirtab{ & \\ & \\ \ \\ \ & \ & \ \\ \ & \ & \ & \ } & 
    \cirtab{ & \\ \ & \\ & \ & \ & \  \\ & & \ & \ & \ & \  \\ & } &
    \cirtab{ & \\ \ \\ \ & \ & \ & \ \\ & \\ \ & \ & \ } &
    \cirtab{ & \\& \ \\ & \ & & \ & \ & \ \\ & \\ & & & \ & \ & \ } &
    \cirtab{ & \\ & & & & \ \\ & \ & & \ & \ & \ \\ & \\ & & \ & & \ & \ } \\
    \scriptstyle \lambda=(0,0,1,3,4)         &
    \scriptstyle \nu/\mu=(0,1,4,6)/(0,0,1,2) &
    \scriptstyle \comp{a}=(0,1,4,0,3)        &
    \scriptstyle w = 13728456                &
    \end{array}
  \end{displaymath}
\caption{\label{fig:diagrams}Examples of diagrams.}
\end{figure}

A \newword{tableau} or \newword{filling} on $D$ is any map $T:D\rightarrow\mathbb{N}$. Abusing notation, we let $D$ also denote the \emph{row tableau} on the shape $D$ where each cell is assigned its row index. For example, Fig.~\ref{fig:tableaux} shows a generic tableau (left) and the row tableau (right).

\begin{figure}[ht]
    \begin{displaymath}
      T = {\vline\tableau{1 \\ 3 & & 3 \\  & 2 & 4 } }
      \hspace{2em}
      D = {\vline\tableau{ 1 \\ 2 & & 2 \\ & 3 & 3 }}
    \end{displaymath}
\caption{\label{fig:tableaux}Two tableaux for a diagram $D$.}
\end{figure}

\begin{definition}
    For $D$ a diagram and $T$ a tableau of shape $D$, set
    \begin{equation}
      \Delta_T = \prod_j ( T^{(j)} \mid D^{(j)} )
    \end{equation}
    where $T^{(j)}$ is the array of values in the $j$th column of $T$, and for arrays $\comp{a}, \comp{b} \in\mathbb{N}^m$, we set $(\comp{a} \mid \comp{b}) = \det(z_{a_i,b_j}) \in \mathbb{C}[z_{i,j}]$.
\end{definition}

For example, taking $T$ to be the left tableau in Fig.~\ref{fig:tableaux}, we have
    \[ \Delta_T =
    \begin{vmatrix}
      z_{11} & z_{12} \\ z_{31} & z_{32} 
    \end{vmatrix}
    \begin{vmatrix}
      z_{23}
    \end{vmatrix}
    \begin{vmatrix}
      z_{32} & z_{33} \\ z_{42} & z_{43}
    \end{vmatrix} . \]
Notice $\Delta_T=0$ whenever $T$ has repeated column entries.

\begin{definition}
  The \newword{Schur module} $\Schur_D$ is the $\mathbb{C}$-span of $\{ \Delta_T \mid T:D \rightarrow [n]\}$.
\end{definition}
  
In particular, taking $D$ to be the Ferrers diagram of the partition $\lambda$ as in the leftmost diagram of Fig.~\ref{fig:diagrams}, the Schur module $\Schur_D$ is the irreducible representation of $\GL_n$ indexed by $\lambda$, which we denote by $\Schur_{\lambda}$.

For $B \subset \GL_n$ the subalgebra of upper triangular matrices, there is a $B$-stable ideal $I = \langle z_{i,j} \mid i>j \rangle$. We consider the $B$-module spanned by those fillings that respect this ideal. Say a filling $T$  is \newword{flagged} if $T_{i,j} \leq i$.

\begin{definition}
    The \newword{flagged Schur module} $\fSchur_D$ is the $B$-module $\Schur_D / (\Schur_D \cap I)$ spanned by $\{ \Delta_T \mid T_{i,j} \leq i \}$.  
\end{definition}  
  
The \newword{character} of a $B$-module is the trace of the action of the matrix $X=\mathrm{diag}(x_1,\ldots,x_n)$. Many flagged Schur modules arise naturally in the context of geometry: for $D$ a key diagram, the module $\fSchur_{D}$ coincides with those introduced by Demazure \cite{Dem74,Dem74a} that give a filtration of irreducible modules with respect to the Weyl group; and for $D$ a Rothe diagram, the module $\fSchur_D$ coincides with the Schubert modules introduced by Kraskiewicz and Pragacz \cite{KP87,KP04} whose characters are Schubert polynomials of Lascoux and Sch{\"u}tzenberger \cite{LS82} that compute intersection multiplicities for the flag manifold.

In what follows, we study these modules through their characters, to wit the following results are useful.

The \newword{weight} of a filling $T$ is the weak composition $\wt(T)$ whose $i$th part is the number of entries of $T$ with label $i$. Given $\comp{a}$, let $x^{\comp{a}} = x_1^{a_1} \cdots x_n^{a_n}$.

\begin{lemma}\label{lem:eigen}
	For a flagged filling $T$ of $D$, $\Delta_T\pmod I$ is an eigenvector with eigenvalue $x_1^{\wt(T)_1}\cdots x_n^{\wt(T)_n}$ under the action of $X=\mathrm{diag}(x_1,\ldots,x_n)$. In particular, the monomial $x^{\wt(T)}$ appears with positive multiplicity in $\ch(\fSchur_D)$.
\end{lemma}

\begin{proof}
	Since each entry in $T^{(j)}$ is no greater than the corresponding entry in $D^{(j)}$, the diagonal of the matrix defining $(T^{(j)} \mid D^{(j)})$ includes only indeterminates $z_{k,\ell}$ for which $k\le \ell$, ensuring that $(T^{(j)} \mid D^{(j)})\notin I$. Since $I$ is prime we can be sure that $\Delta_T\notin I$ as well. Namely $\Delta_T$ is not zero in $\fSchur_D$.
	
	Note that $X\cdot z_{k,\ell}=x_{k}z_{k,\ell}$ for any $k,\ell$. The $i$th row vector of the matrix of $(T^{(j)} \mid D^{(j)})$ contains only indeterminates $z_{k,\ell}$ with $k=T^{(j)}_i$. Hence $X\cdot (T^{(j)} \mid D^{(j)})$ is the result of multiplying each row $i$ in $(T^{(j)} \mid D^{(j)})$ by $x_{T^{(j)}_i}$. Multi-linearity of the determinant yields $X\cdot (T^{(j)} \mid D^{(j)})=x^{\wt(T^{(j)})}(T^{(j)} \mid D^{(j)})$. Therefore, 
	\[
	X\cdot\Delta_T=\prod_j x^{\wt(T^{(j)})}(T^{(j)} \mid D^{(j)})=x^{\wt(T)}\prod_j(T^{(j)} \mid D^{(j)})=x^{\wt(T)}\Delta_T.
	\]
	Take a set of flagged fillings which indexes a basis for $\fSchur_D$ in the natural way. The basis is in fact an eigenbasis for the action of a diagonal matrix and all eigenvalues are positive monomials. Since $x^{\wt(T)}$ is an eigenvalue, the result follows.
\end{proof}

Th following is used to determine when the character of a flagged Schur module differs from the Kohnert polynomial (defined below). Given position integers $r<s$, let $\alpha_{r,s}$ be the weak composition with $r$th entry $1$, $s$th entry $-1$, and all others $0$.

\begin{proposition}\label{prop:A1}
	Let $D$ be a diagram and $C$ a set of column indices in which row $s$ contains a cell but row $r<s$ does not. Then $x^{\wt(D)+|C|\alpha_{r,s}}$ occurs in the monomial expansion of $\ch(\fSchur_D)$.
\end{proposition}

\begin{proof}
	We will define a flagged filling $T$ on $D$. For $j\notin C$ we simply let $T^{(j)}=D^{(j)}$. Otherwise, $D^{(j)}=(a_1,\ldots, a_m)$ contains the entry $s$ but not $r$. Say $k$ is the least index so that $a_k>r$ and $a_\ell=s$. Then define
	\[
	T^{(j)}=(\ldots,a_{k-1},r,a_k,a_{k+1},\ldots, a_{\ell-1},a_{\ell+1},\ldots ).
	\]
	That is, the $k$th entry becomes $r$ and entries of index $i$ with $k<i\le\ell$ become $a_{i-1}$ with all other entries the same as in $D^{(j)}$. The second case is vacuous if $k=\ell$.
	
	Since the entries in $T$ are no greater than the corresponding entries in $D$, $T$ is indeed a flagged filling. It is clear to see that $\wt(T^{(j)})=\wt(D^{(j)})+\alpha_{r,s}$ if $j\in C$, and $\wt(T^{(j)})=\wt(D^{(j)})$ otherwise. Therefore $\wt(T)=\wt(D)+|C|\alpha_{r,s}$. The proof is completed by Lemma~\ref{lem:eigen}.
\end{proof}

\subsection{Demazure modules and characters}
\label{sec:modules-characters}

Each irreducible $\GL_n$-module $\Schur_{\lambda}$ decomposes into weight spaces as $\Schur_{\lambda} = \bigoplus_{\comp{a}} \Schur_{\lambda}^\comp{a}$. Demazure \cite{Dem74a} considered the $B$-action on the \emph{extremal weight spaces} of $S_{\lambda}$, those whose weak composition weight $\comp{a}$ is a rearrangement of the highest weight $\lambda$. The extremal weights are naturally indexed by the pair $(\lambda,w)$, where $\lambda$ is the partition weight and $w$ is the minimum length permutation that rearranges $\lambda$ to $\comp{a}$.

\begin{definition}[\cite{Dem74a}]
    For a partition $\lambda$ and a permutation $w$, the \newword{Demazure module} $\Schur_{\lambda}^w$ is the $B$-submodule of $\Schur_{\lambda}$ generated by the weight space $\Schur_{\lambda}^{w\cdot\lambda}$. 
\end{definition}

At the one extreme, the Demazure module $\Schur_{\lambda}^{\mathrm{id}}$ for the identity permutation is the one-dimensional subspace of $\Schur_{\lambda}$ containing the highest weight element. At the other extreme, the Demazure module $\Schur_{\lambda}^{w_0}$ for the long element of $\mathcal{S}_n$ is the full module $\Schur_{\lambda}$. In general, the Demazure modules give a filtration from the highest weight element, namely $\Schur_{\lambda}^{u} \subset \Schur_{\lambda}^{w}$ whenever $u\prec w$ in Bruhat order.

The characters of Demazure modules can be computed by degree-preserving divided difference operators \cite{Dem74}, as well as myriad combinatorial models.

For a positive integer $i$, the \newword{divided difference operator} $\partial_i$ is the linear operator that acts on polynomials $f \in \mathbb{Z}[x_1,x_2,\ldots]$ by
\begin{equation}
  \partial_i(f) = \frac{f - s_i \cdot f}{x_i - x_{i+1}} ,
\end{equation}
where $s_i \in \mathcal{S}_{\infty}$ is the simple transposition that exchanges $x_i$ and $x_{i+1}$. Fulton \cite{Ful92} connected the divided difference operators with modern intersection theory, providing a direct geometric context for Schubert polynomials \cite{LS82}. An isobaric variation, sometimes called the \newword{Demazure operator}, is defined by
\begin{equation}
  \pi_i(f) = \partial_i (x_i f).
\end{equation}

Given a permutation $w$, we may define
\begin{eqnarray*}
  \partial_w  & = & \partial_{s_{i_k}} \cdots \partial_{s_{i_1}}, \\
  \pi_w  & = & \pi_{s_{i_k}} \cdots \pi_{s_{i_1}}, 
\end{eqnarray*}
for any expression $s_{i_k} \cdots s_{i_1} = w$ with $k$ minimal. Such an expression is called a \newword{reduced expression for $w$}. Both $\partial_i$ and $\pi_i$ satisfy the braid relations, and so are independent of the choice of reduced expression for $w$. In fact, the  expression for $\pi_w$ need not be reduced, since $\pi_i(f) = f$ whenever $s_i \cdot f = f$.

\begin{theorem}[\cite{Dem74a}]
  The character of the Demazure module $\Schur_{\lambda}^w$ is
  \begin{equation}
      \ch(\Schur_{\lambda}^w) = \pi_w (x_1^{\lambda_1} \cdots x_n^{\lambda_n}).
  \end{equation}
\end{theorem}

In particular, when $s_i w \prec w$ in Bruhat order, we have the following identity,
\begin{equation} \label{e:si}
  \ch(\Schur_{\lambda}^w) = \pi_i \left( \ch(\Schur_{\lambda}^{s_i w}) \right) .
\end{equation}

For $\D(\comp{a})$ the key (left-justified) diagram of a composition $\comp{a} = w \cdot \lambda$, we have 
\begin{equation}
    \ch(\fSchur_{\D(\comp{a})}) = \ch(\Schur_{\lambda}^w).
\end{equation}
In this sense, the flagged Schur modules generalize Demazure characters. 

Both key diagrams indexing Demazure modules and Rothe diagrams indexing Kraskiewicz--Pragacz modules exhibit the following property.

\begin{definition}\label{def:northwest}
  A diagram $D$ is \newword{northwest} if whenever $(j,k),(i,l)\in D$ with $i<j$ and $k<l$, then $(i,k)\in D$.
\end{definition}

Northwest in this context is equivalent to \newword{southwest} as used by Assaf and Searles \cite{AS19,Ass-KC} for Kohnert polynomials.

Magyar \cite{Mag98-2} gave a recurrence for computing characters of flagged Schur modules for  \%-avoiding shapes (see Definition~\ref{def:pa}). To state the recurrence, let $C_k$ denote the diagram with cells in rows $1,2,\ldots,k$ of the first column. Given a diagram $D$, let $s_r D$ denote the diagram obtained from $D$ by permuting the cells in rows $r,r+1$. Restricting Magyar's result to northwest shapes gives the following.

\begin{theorem}\label{thm:recur}
  The characters of the flagged Schur modules of northwest shapes satisfy the following recurrence:
  \begin{itemize}
      \item[(M1)] $\ch(\fSchur_{\varnothing}) = 1$;
      \item[(M2)] if the first column of $D$ is exactly $C_k$, then 
      \begin{equation}\ch(\fSchur_{D}) = x_1 x_2 \cdots x_k \ch(\fSchur_{D-C_k});\end{equation}
      \item[(M3)] if every cell in row $r$ has a cell in row $r+1$ in the same column, then \begin{equation}\label{e:M3}\ch(\fSchur_{D}) = \pi_r \left(\ch(\fSchur_{s_r D})\right).\end{equation}
  \end{itemize}
\end{theorem}

\begin{figure}[ht]
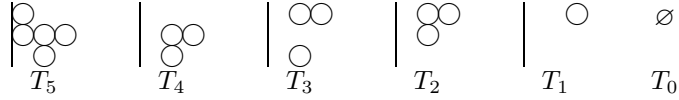

  \begin{displaymath}
    \arraycolsep=1.5\cellsize
  \begin{array}{cccccc}
    \cirtab{ \ \\ \ & \ & \ \\ & \ } & 
    \cirtab{ & \\  & \ & \ \\ & \ & } &
    \cirtab{ & \ & \ \\ & \\ & \ & } &
    \cirtab{ & \ & \ \\ & \ & \\ & } &
    \cirtab{ & & \ \\ & \\ & } &
    \varnothing \\
    T_5 &
    T_4 &
    T_3 &
    T_2 &
    T_1 &
    T_0 
  \end{array}
  \end{displaymath}
\caption{\label{fig:peeling} Applying Magyar's recurrence to a northwest diagram (left) to compute the character of the flagged Schur module. }
\end{figure}

For example, Magyar's recurrence allows us to compute the character for the flagged Schur module indexed by the leftmost diagram $T_5$ in Figure~\ref{fig:peeling} as follows:
\begin{displaymath}
\begin{array}{rll}
  \ch(\fSchur_{T_5}) = & x_1 x_2 \cdot \ch(\fSchur_{T_4}) & \text{(M2)} \\
   = & x_1 x_2 \cdot \pi_1 (\ch(\fSchur_{T_3}))  & \text{(M3)} \\
   = & x_1 x_2 \cdot \pi_1 ( \pi_2 (\ch(\fSchur_{T_2}) ) )  & \text{(M3)} \\
   = & x_1 x_2 \cdot \pi_1 ( \pi_2 (x_1 x_2 \cdot \ch(\fSchur_{T_1})) ) )  & \text{(M2)} \\
   = & x_1 x_2 \cdot \pi_1 ( \pi_2 (x_1 x_2 \cdot ( x_1 \cdot \ch(\fSchur_{T_0}))) ) )  & \text{(M2)} \\
   = & x_1 x_2 \cdot \pi_1 ( \pi_2 (x_1^2 x_2 \cdot ( 1 ) ) )  & \text{(M1)} \\
   = & x_1 x_2 \cdot \pi_1 ( x_1^2 x_2 + x_1^2 x_3 )  & \\
   = & x_1 x_2 \cdot ( x_1^2 x_2 + x_1^2 x_3 + x_1 x_2^2 + x_1 x_2 x_3 + x_2^2 x_3 ) & \\
   = & x_1^3 x_2^2 + x_1^3 x_2 x_3 + x_1^2 x_2^3 + x_1^2 x_2^2 x_3 + x_1 x_2^3 x_3  & 
\end{array}
\end{displaymath}

Two Demazure modules $\Schur_{\mu}^u$ and $\Schur_{\nu}^v$ correspond precisely when $u \cdot \mu = v \cdot \nu$. While this implies $\mu=\nu$, the permutations $u,v$ can differ. Therefore the more natural indexing set for Demazure characters is obtained by specifying the weak composition given by the permutation acting on the partition. Define
\begin{equation}
    \key_{\comp{a}} = \ch(\Schur_{\lambda}^w),
\end{equation}
where $w$ sorts the weak composition $\comp{a}$ to the weakly decreasing partition $\lambda$.

Perhaps unsurprising when comparing \eqref{e:si} and \eqref{e:M3}, Reiner and Shimozono \cite{RS98} use Magyar's recurrence for the character of flagged Schur modules \cite{Mag98-2} to prove the characters of flagged Schur modules decompose as a nonnegative sum of Demazure characters; see \cite{RS98} for precise definitions.

\begin{theorem}[\cite{RS98}]\label{thm:RS}
  For $D$ a northwest diagram, we have
  \[ \ch(\fSchur_D) = \sum_{\comp{a}} c^{D}_{\comp{a}} \key_{\comp{a}} , \]
  where $c^{D}_{\comp{a}}$ is the number of \emph{$D$-peelable tableaux} whose \emph{left-nil key} is $\comp{a}$. 
\end{theorem}

One consequence of our main result is a new proof of this positivity and a new combinatorial interpretation for the coefficients.

\subsection{Kohnert polynomials}
\label{sec:modules-kohnert}

Kohnert gave a combinatorial model for Demazure characters \cite{Koh91} in terms of the following operation on diagrams.

\begin{definition}[\cite{Koh91}]
  A \newword{Kohnert move} on a diagram selects the rightmost cell of a given row and moves the cell up, staying within its column, to the first available position above, if it exists, jumping over other cells in its way as needed. 
\label{def:kohnert_move}
\end{definition}

Given a diagram $D$, let $\KD(D)$ denote the set of diagrams that can be obtained by some sequence of Kohnert moves on $D$. For example, Fig.~\ref{fig:KohnertPosetNW} shows the Kohnert diagrams for the bottom diagram, where an edge from $S$ up to $T$ indicates $T$ can be obtained from $S$ by a single Kohnert move. 

\begin{figure}[ht]
 \begin{center}
    \begin{tikzpicture}[xscale=1.75,yscale=1.6]
    \node at (3,1)     (A) {$\cirtab{ & \\ ~ & \\ ~ & ~ \\ & ~ \\}$};
    \node at (2,2)     (B) {$\cirtab{ ~ & \\ & \\ ~ & ~ \\ & ~ \\}$};
    \node at (3,2)     (C) {$\cirtab{ & \\ ~ & ~ \\ ~ & \\ & ~ \\}$};
    \node at (4,3)     (D) {$\cirtab{ & \\ ~ & ~ \\ ~ & ~ \\ & \\}$};   
    \node at (2,3)     (E) {$\cirtab{ ~ & \\ & ~ \\ ~ & \\ & ~ \\}$};
    \node at (3,4)     (F) {$\cirtab{ ~ & \\ & ~ \\ ~ & ~ \\ & \\}$};
    \node at (3,3)     (G) {$\cirtab{ & ~ \\ ~ & \\ ~ & \\ & ~ \\}$};
    \node at (2,4)     (H) {$\cirtab{ ~ & \\ ~ & ~ \\ & \\ & ~ \\}$};
    \node at (1,4)     (J) {$\cirtab{ ~ & ~ \\ & \\ ~ & \\ & ~ \\}$};
    \node at (2,6)     (K) {$\cirtab{ ~ & ~ \\ & ~ \\ ~ & \\ & \\}$};
    \node at (4,4)     (L) {$\cirtab{ & ~ \\ ~ & \\ ~ & ~ \\ & \\}$};
    \node at (5,5)     (M) {$\cirtab{ & ~ \\ ~ & ~ \\ ~ & \\ & \\}$};
    \node at (2,5)     (N) {$\cirtab{ ~ & ~ \\ ~ & \\ & \\ & ~ \\}$};
    \node at (3,5)     (O) {$\cirtab{ ~ & ~ \\ & \\ ~ & ~ \\ & \\}$};
    \node at (4,5)     (P) {$\cirtab{ ~ & \\ ~ & ~ \\ & ~ \\ & \\}$};
    \node at (3,6)     (Q) {$\cirtab{ ~ & ~ \\ ~ & \\ & ~ \\ & \\}$};
    \node at (3,7)     (S) {$\cirtab{ ~ & ~ \\ ~ & ~ \\ & \\ & \\}$};
    \draw[thin] (A) -- (B) ;
    \draw[thin] (A) -- (C) ;
    \draw[thin] (A) -- (D) ;
    \draw[thin] (B) -- (E) ;
    \draw[thin] (B) -- (F) ;
    \draw[thin] (C) -- (G) ;
    \draw[thin] (C) -- (H) ;
    \draw[thin] (C) -- (D) ;
    \draw[thin] (D) -- (L) ;
    \draw[thin] (D) -- (M) ;
    \draw[thin] (E) -- (F) ;
    \draw[thin] (E) -- (J) ;
    \draw[thin] (E) -- (H) ;
    \draw[thin] (F) -- (K) ;
    \draw[thin] (F) -- (O) ;
    \draw[thin] (G) -- (J) ;
    \draw[thin] (G) -- (L) ;
    \draw[thin] (G) -- (N) ;
    \draw[thin] (H) -- (N) ;
    \draw[thin] (H) -- (P) ;
    \draw[thin] (L) -- (O) ;
    \draw[thin] (L) -- (M) ;
    \draw[thin] (J) -- (N) ;
    \draw[thin] (J) -- (O) ;
    \draw[thin] (K) -- (S) ;
    \draw[thin] (L) -- (M) ;
    \draw[thin] (L) -- (O) ;
    \draw[thin] (M) -- (S) ;
    \draw[thin] (N) -- (Q) ;
    \draw[thin] (O) -- (K) ;
    \draw[thin] (P) -- (Q) ;
    \draw[thin] (P) -- (S) ;
    \draw[thin] (Q) -- (S) ;
    \end{tikzpicture}
    \caption{\label{fig:KohnertPosetNW}The poset of Kohnert diagrams for the bottom diagram.}
  \end{center}
\end{figure}
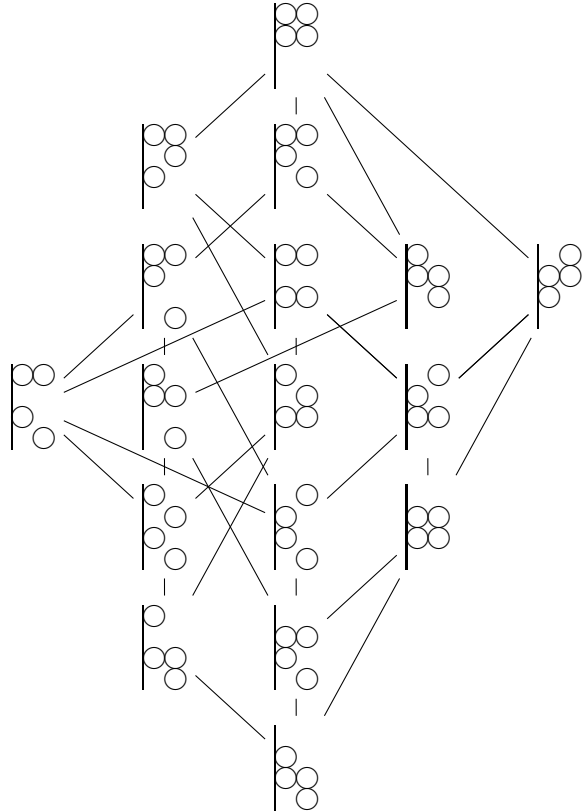

\begin{theorem}[\cite{Koh91}]
  The Demazure character $\key_{\comp{a}}$ is 
  \[ \key_{\comp{a}} = \sum_{T \in \KD(\D(\comp{a}))} x_1^{\wt(T)_1} \cdots x_n^{\wt(T)_n}, \]
  where the \newword{weight} of a diagram $D$, denoted by $\wt(D)$, is the weak composition whose $i$th part is the number of cells of $D$ in row $i$.
  \label{thm:kohnert}
\end{theorem}

Assaf and Searles define the Kohnert polynomial of any diagram as follows.

\begin{definition}[\cite{AS19}]
  The \newword{Kohnert polynomial} for a diagram $D$ is 
  \[ \kohnert_{D} = \sum_{T \in \KD(D)} x_1^{\wt(T)_1} \cdots x_n^{\wt(T)_n}. \]
  By convention, set $\kohnert_{\varnothing}=1$ for the empty diagram.
\end{definition}

Assaf and Searles conjectured \cite{AS19} and Assaf proved \cite{Ass-KC} that for $D$ northwest, the Kohnert polynomial expands nonnegatively into Demazure characters. 

\begin{theorem}[\cite{Ass-KC}]\label{thm:kohnert-dem}
  For $D$ northwest, we have
  \[ \kohnert_D = \sum_{\comp{a}} c^D_{\comp{a}} \key_a, \]
  where $c^{D}_{\comp{a}}$ is the number of \emph{Yamanouchi Kohnert diagrams} for $D$ of weight $\comp{a}$. 
\end{theorem}

For the example, the Kohnert polynomial expands into Demazure characters as
\[ \kohnert_{D} = \key_{(0,1,2,1)} + \key_{(0,2,2,0)} . \]

Comparing Theorem~\ref{thm:RS} with Theorem~\ref{thm:kohnert-dem}, one begins to suspect the following.

\begin{conjecture}[\cite{AS19}]\label{conj:main}
  For $D$ a northwest diagram, we have
  \begin{equation}
    \ch(\fSchur_{D}) = \kohnert_{D} .      
  \end{equation}
\end{conjecture}

We prove Conjecture~\ref{conj:main} by showing Kohnert polynomials of northwest diagrams satisfy the recurrence in Theorem~\ref{thm:recur}.

By definition, we have $\kohnert_{\varnothing} = 1$, establishing Theorem~\ref{thm:recur}(M1). It is straightforward to establish Theorem~\ref{thm:recur}(M2).

\begin{lemma}\label{lem:tower}
Let $D$ be a northwest diagram such that the first column is exactly $C_k$. Then $D-C_k$ is also northwest, and we have
\begin{equation} 
  \kohnert_{D} = x_1 x_2 \cdots x_k \kohnert_{D-C_k}.
\end{equation}
\end{lemma}

\begin{proof}
    As Kohnert moves preserve the columns of cells and no moves are possible in column $1$ of $D$, all diagrams in $\KD(D)$ must coincide in column $1$ with cells in rows $1,\ldots, k$. In particular, since column $1$ is the leftmost column and admits no Kohnert moves, there is a poset isomorphism between $\KD(D)$ and $\KD(D-C_k)$ obtained by deleting all cells in column $1$. Since those cells contribute $x_1 x_2 \cdots x_k$ to the weight, the result follows.
\end{proof}

We next consider $D$ and $s_i D$ when row $i$ is a subset of row $i+1$, meaning the set of occupied columns in row $i$ is a subset of the occupied columns of row $i+1$. These conditions combined imply  all cells in row $i+1$ without cells above them in row $i$ lie to the right of all cells in row $i$. Otherwise, the cell in row $i+1$ without a cell above it in row $i$ and any cell to the right in row $i$ violate the northwest condition. This also implies $s_iD$ is northwest. The only violation must take place in rows $i$ and $i+1$, but we only move cells in row $i+1$ to row $i$ where all cells to the left in row $i+1$ have a cell in row $i$ above them. 

\begin{lemma}\label{lem:subset}
If $D$ is a northwest diagram and row $i$ is contained in row $i+1$, then $s_i D$ is also northwest and $s_iD \in \KD(D)$. In particular, $\KD(s_i D) \subseteq \KD(D)$. 
\end{lemma}

\begin{proof}
Let $c$ be the rightmost occupied column in row $i$ of $D$. Let $m$ be the number of cells in row  $i+1$ strictly to the right of column $c$. By choice of $c$, we may apply a sequence of $m$ Kohnert moves to row $i+1$ of $D$, resulting in all cells strictly to right of column $c$ moving from row $i+1$ to row $i$. Since row $i$ is a subset of row $i+1$, any column weakly left of column $c$ either has no cells in rows $i,i+1$ or cells in both rows $i,i+1$. Thus this sequence of Kohnert moves yields $s_i D$.
\end{proof}

Since $\pi_i$ acts on a polynomial by adding additional monomials, this makes Theorem~\ref{thm:recur}(M3) plausible. However, in order to understand precisely how $\pi_i$ acts on Kohnert polynomials, we must rely on Demazure crystals.

%
\section{Crystal graphs}
%
\label{sec:crystals}

In order to understand the action of the degree-preserving divided difference operator on Kohnert polynomials, we shift our paradigm to Demazure operators on Demazure subsets of highest weight crystals.

\subsection{Tableaux crystals}

A \newword{crystal graph} \cite{Kas90} is a combinatorial model for a highest weight module, consisting of a vertex set $\B$, corresponding to the crystal base (see also the canonical basis \cite{Lus90}), and directed, colored edges, corresponding to deformations of the Chevalley generators. For a highest weight $\lambda$ for $\GL_n$, the crystal base for $\Schur_{\lambda}$ is naturally indexed by $\SSYT_n(\lambda)$, the \newword{semistandard Young tableaux} of shape $\lambda$ with entries in $\{1,2,\ldots,n\}$. We adopt the French (coordinate) convention for $\SSYT$ in which entries weakly increase left to right within rows and strictly increase bottom to top within columns. Kashiwara and Nakashima \cite{KN94} and Littelmann \cite{Lit95} defined the crystal edges on tableaux as follows.

\begin{definition}\label{def:pair-SSYT}
  For $T \in \SSYT_n(\lambda)$ and $1 \leq i <n$, we \newword{i-pair} the cells of $T$ with entries $i$ or $i+1$ iteratively by $i$-pairing an unpaired $i+1$ with an unpaired $i$ weakly to its right whenever all entries $i$ or $i+1$ lying between them are already $i$-paired.
\end{definition}

The raising and lowering operators on crystals are maps $e_i,f_i:\B \rightarrow \B\cup\{0\}$.

\begin{definition}\label{def:raise-SSYT}
  For $1 \leq i <n$, the \newword{crystal raising operator} $e_i$ acts on $T \in \SSYT_n(\lambda)$ as follows:
  \begin{itemize}
  \item if all entries $i+1$ of $T$ are $i$-paired, then $e_i(T)=0$;
  \item otherwise, $e_i$ changes the leftmost unpaired $i+1$ to $i$.
  \end{itemize}
\end{definition}

\begin{definition}\label{def:lower-SSYT}
  For $1 \leq i <n$, the \newword{crystal lowering operator} $f_i$ acts on $T \in \SSYT_n(\lambda)$ as follows:
  \begin{itemize}
  \item if all entries $i$ of $T$ are $i$-paired, then $f_i(T)=0$;
  \item otherwise, $f_i$ changes the rightmost unpaired $i$ to $i+1$.
  \end{itemize}
\end{definition}

We draw a crystal graph with an $i$-colored edge from $b$ to $b'$ whenever $b' = f_i(b)$, omitting edges to $0$. For examples, see Fig.~\ref{fig:TableauxCrystals}. 

\begin{figure}[ht]
 \begin{center}
    \begin{tikzpicture}[xscale=1.75,yscale=1.5]
    \node at (2,7)     (A35) {$\tableau{3 \\ 2 \\ 1 & 1}$};
    \node at (1,6)     (A14) {$\tableau{3 \\ 2 \\ 1 & 2}$};
    \node at (3,6)     (A54) {$\tableau{4 \\ 2 \\ 1 & 1}$};
    \node at (1,5)     (A13) {$\tableau{3 \\ 2 \\ 1 & 3}$};
    \node at (2,5)     (A33) {$\tableau{4 \\ 2 \\ 1 & 2}$};
    \node at (3,5)     (A53) {$\tableau{4 \\ 3 \\ 1 & 1}$};
    \node at (1.5,4)   (A22) {$\tableau{3 \\ 2 \\ 1 & 4}$};
    \node at (2,4)     (A32) {$\tableau{4 \\ 2 \\ 1 & 3}$};
    \node at (2.5,4)   (A42) {$\tableau{4 \\ 3 \\ 1 & 2}$};
    \node at (1,3)     (A11) {$\tableau{4 \\ 3 \\ 2 & 2}$};
    \node at (2,3)     (A31) {$\tableau{4 \\ 3 \\ 1 & 3}$};
    \node at (3,3)     (A51) {$\tableau{4 \\ 2 \\ 1 & 4}$};
    \node at (1,2)     (A10) {$\tableau{4 \\ 3 \\ 2 & 3}$};
    \node at (3,2)     (A50) {$\tableau{4 \\ 3 \\ 1 & 4}$};
    \node at (2,1)     (Axx) {$\tableau{4 \\ 3 \\ 2 & 4}$};
    \draw[thick,blue  ,->](A35) -- (A14)   node[midway,above]{$1$};
    \draw[thick,blue  ,->](A54) -- (A33)   node[midway,above]{$1$};
    \draw[thick,blue  ,->](A53) -- (A42)   node[midway,above]{$1$};
    \draw[thick,blue  ,->](A42) -- (A11)   node[midway,above]{$1$};
    \draw[thick,blue  ,->](A31) -- (A10)   node[midway,above]{$1$};
    \draw[thick,blue  ,->](A50) -- (Axx)   node[midway,above]{$1$};
    \draw[thick,purple,->](A14) -- (A13)   node[midway,left ]{$2$};
    \draw[thick,purple,->](A54) -- (A53)   node[midway,left ]{$2$};
    \draw[thick,purple,->](A33) -- (A32)   node[midway,left ]{$2$};
    \draw[thick,purple,->](A32) -- (A31)   node[midway,left ]{$2$};
    \draw[thick,purple,->](A11) -- (A10)   node[midway,left ]{$2$};
    \draw[thick,purple,->](A51) -- (A50)   node[midway,left ]{$2$};
    \draw[thick,violet,->](A35) -- (A54)   node[midway,above]{$3$};
    \draw[thick,violet,->](A14) -- (A33)   node[midway,above]{$3$};
    \draw[thick,violet,->](A13) -- (A22)   node[midway,above]{$3$};
    \draw[thick,violet,->](A22) -- (A51)   node[midway,above]{$3$};
    \draw[thick,violet,->](A31) -- (A50)   node[midway,above]{$3$};
    \draw[thick,violet,->](A10) -- (Axx)   node[midway,above]{$3$};
    \node at (4,4)  (e1) {$\tableau{3 & 3 \\ 2 & 2}$};
    \node at (4.5,6)  (c2) {$\tableau{2 & 3 \\ 1 & 2}$};
    \node at (4.5,5)  (d2) {$\tableau{3 & 3 \\ 1 & 2}$};
    \node at (4.5,3)  (f2) {$\tableau{3 & 4 \\ 2 & 2}$};
    \node at (4.5,2)  (g2) {$\tableau{3 & 4 \\ 2 & 3}$};
    \node at (5.5,8)  (a3) {$\tableau{2 & 2 \\ 1 & 1}$};
    \node at (5.5,7)  (b3) {$\tableau{2 & 3 \\ 1 & 1}$};
    \node at (5.5,6)  (c3) {$\tableau{3 & 3 \\ 1 & 1}$};
    \node at (5.5,5)  (d3) {$\tableau{2 & 4 \\ 1 & 2}$};
    \node at (5.25,4)(l3) {$\tableau{2 & 4 \\ 1 & 3}$};
    \node at (5.75,4)(r3) {$\tableau{3 & 4 \\ 1 & 2}$};
    \node at (5.5,3)  (f3) {$\tableau{3 & 4 \\ 1 & 3}$};
    \node at (5.5,2)  (g3) {$\tableau{4 & 4 \\ 2 & 2}$};
    \node at (5.5,1)  (h3) {$\tableau{4 & 4 \\ 2 & 3}$};
    \node at (5.5,0)  (i3) {$\tableau{4 & 4 \\ 3 & 3}$};
    \node at (6.5,6)  (c4) {$\tableau{2 & 4 \\ 1 & 1}$};
    \node at (6.5,5)  (d4) {$\tableau{3 & 4 \\ 1 & 1}$};
    \node at (6.5,3)  (f4) {$\tableau{4 & 4 \\ 1 & 2}$};
    \node at (6.5,2)  (g4) {$\tableau{4 & 4 \\ 1 & 3}$};
    \node at (7,4)  (e5) {$\tableau{4 & 4 \\ 1 & 1}$};
    \draw[thick,blue  ,->] (b3) -- (c2)  node[midway,above]{$1$};
    \draw[thick,blue  ,->] (c3) -- (d2)  node[midway,above]{$1$};
    \draw[thick,blue  ,->] (c4) -- (d3)  node[midway,above]{$1$};
    \draw[thick,blue  ,->] (d2) -- (e1)  node[midway,above]{$1$};
    \draw[thick,blue  ,->] (d4) -- (r3)  node[midway,above]{$1$};
    \draw[thick,blue  ,->] (r3) -- (f2)  node[midway,above]{$1$};
    \draw[thick,blue  ,->] (e5) -- (f4)  node[midway,above]{$1$};
    \draw[thick,blue  ,->] (f3) -- (g2)  node[midway,above]{$1$};
    \draw[thick,blue  ,->] (f4) -- (g3)  node[midway,above]{$1$};
    \draw[thick,blue  ,->] (g4) -- (h3)  node[midway,above]{$1$};
    \draw[thick,purple,->] (a3) -- (b3)   node[midway,left ]{$2$};
    \draw[thick,purple,->] (b3) -- (c3)   node[midway,left ]{$2$};
    \draw[thick,purple,->] (c2) -- (d2)   node[midway,left ]{$2$};
    \draw[thick,purple,->] (c4) -- (d4)   node[midway,left ]{$2$};
    \draw[thick,purple,->] (d3) -- (l3)   node[midway,left ]{$2$};
    \draw[thick,purple,->] (l3) -- (f3)   node[midway,left ]{$2$};
    \draw[thick,purple,->] (f2) -- (g2)   node[midway,left ]{$2$};
    \draw[thick,purple,->] (f4) -- (g4)   node[midway,left ]{$2$};
    \draw[thick,purple,->] (g3) -- (h3)   node[midway,left ]{$2$};
    \draw[thick,purple,->] (h3) -- (i3)   node[midway,left ]{$2$};
    \draw[thick,violet,->] (b3) -- (c4)   node[midway,above]{$3$};
    \draw[thick,violet,->] (c2) -- (d3)   node[midway,below]{$3$};
    \draw[thick,violet,->] (c3) -- (d4)   node[midway,below]{$3$};
    \draw[thick,violet,->] (d2) -- (r3)   node[midway,above]{$3$};
    \draw[thick,violet,->] (d4) -- (e5)   node[midway,above]{$3$};
    \draw[thick,violet,->] (e1) -- (f2)   node[midway,above]{$3$};
    \draw[thick,violet,->] (r3) -- (f4)   node[midway,above]{$3$};
    \draw[thick,violet,->] (f2) -- (g3)   node[midway,below]{$3$};
    \draw[thick,violet,->] (f3) -- (g4)   node[midway,below]{$3$};
    \draw[thick,violet,->] (g2) -- (h3)   node[midway,above]{$3$};
    \end{tikzpicture}
    \caption{\label{fig:TableauxCrystals}The tableaux crystals for $\GL_4$ with highest weights $(2,1,1,0)$ and $(2,2,0,0)$.}
  \end{center}
\end{figure}
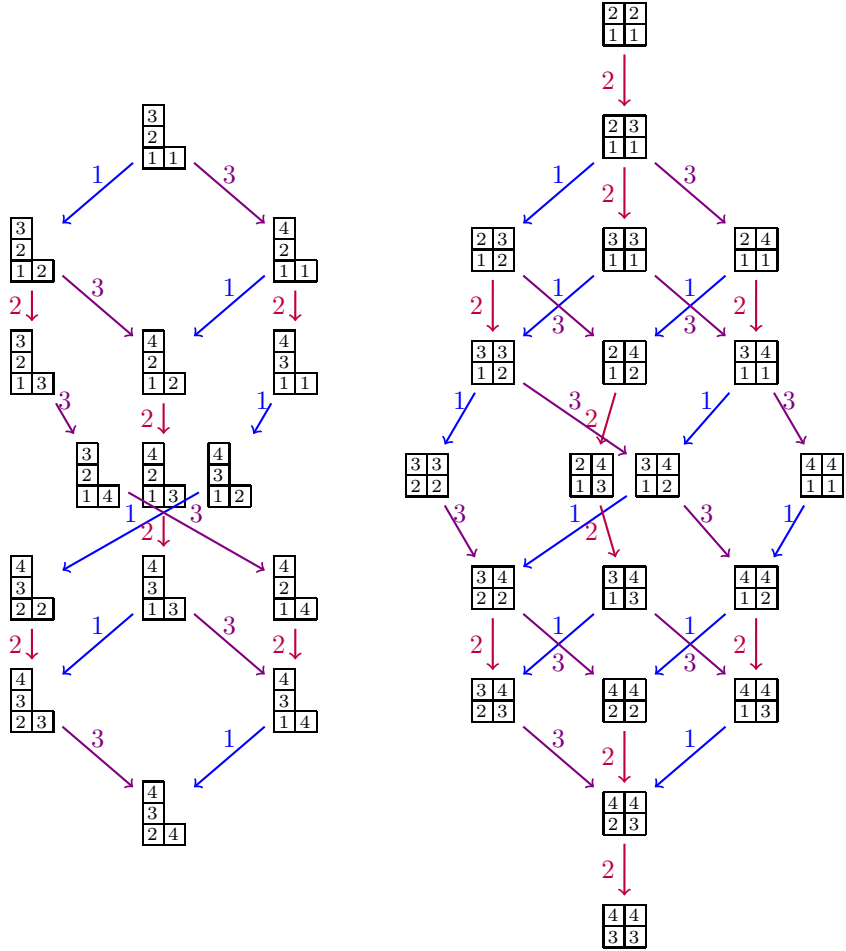

The crystal base is endowed with a map, $\wt:\B\rightarrow\mathbb{Z}^n$, to the weight lattice. Using this, we define the \newword{character} of a crystal $\B$ by
\begin{equation}\label{e:char-crystal}
    \ch(\B) = \sum_{b\in\B} x_1^{\wt(b)_1} \cdots x_n^{\wt(b)_n} .
\end{equation}
Thus if $\B$ is the crystal for a module $V$, then we have $\ch(\B) = \ch(V)$. In particular, for $\B_{\lambda}$ the crystal on $\SSYT_n(\lambda)$, we have
\begin{equation}\label{e:char-schur}
    \ch(\B_{\lambda}) = \ch(\Schur_{\lambda}) = s_{\lambda}(x_1,\ldots,x_n),
\end{equation}
where the latter is the \newword{Schur polynomial} indexed by $\lambda$ in $n$ variables.

\begin{definition}\label{def:hwt}
  An element $u\in\B$ of a highest weight crystal $\B$ is a \newword{highest weight element} if $e_i(u)=0$ for all $1 \leq i < n$.
\end{definition}

Each connected crystal $\B_{\lambda}$ has a unique highest weight element $b$, and we have $\wt(b)=\lambda$. Part of the beauty of highest weight elements, and indeed of crystals, lies in the following decomposition combining \eqref{e:char-crystal} and \eqref{e:char-schur},
\begin{equation}
    \ch(\B) = \sum_{\substack{u \in \B \\ e_i(u)=0 \forall i }} s_{\wt(u)}(x_1,\ldots,x_n).
\end{equation}

\subsection{Demazure crystals}

Littelmann \cite{Lit95} conjectured a crystal structure for Demazure modules as certain truncations of highest weight crystals. Kashiwara \cite{Kas93} proved the result, giving a new proof of the Demazure character formula.

\begin{definition}\label{def:demazure-op}
Given a subset $X$ of a highest weight crystal $\B$ and an index $1\leq i < n$, the \newword{Demazure operator} $\Dem_i$ is given by
\begin{equation}
  \Dem_i (X) = \{ b \in \B \mid e_i^k(b) \in X \text{ for some } k \geq 0 \}.
  \label{e:D}
\end{equation}
\end{definition}

These operators satisfy the braid relations, and so we may define
\begin{equation}
  \Dem_w = \Dem_{i_m} \cdots \Dem_{i_1}
  \label{e:Dw}
\end{equation}
for any expression $s_{i_m} \cdots s_{i_1}$ for the permutation $w$. As with $\pi_w$, the indexing expression for $\Dem_w$ need not be reduced, since $\Dem_i^2 = \Dem_i$.

\begin{definition}[\cite{Lit95}]\label{def:demazure-crystal}
  Given a highest weight crystal $\B(\lambda)$ and a permutation $w$, the \newword{Demazure crystal} $\B^w_{\lambda}$ is given by
  \begin{equation}
    \B^w_{\lambda} = \Dem_{w} (\{ u_{\lambda} \}) ,
    \label{e:BwL}
  \end{equation}
  where $u_{\lambda}$ is the highest weight element of $\B_{\lambda}$.
\end{definition}

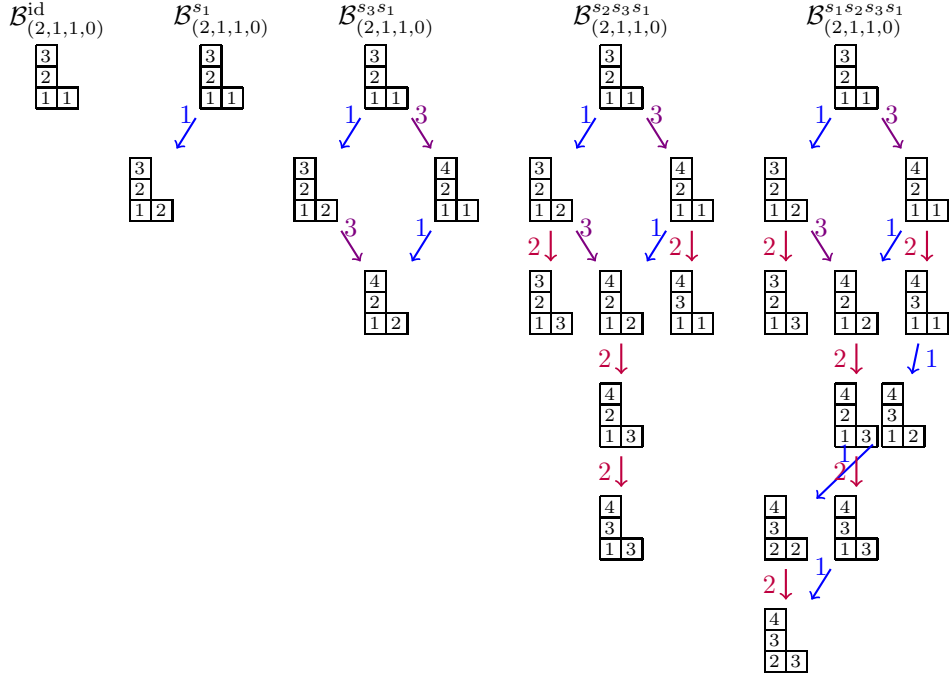
\begin{figure}[ht]
 \begin{center}
    \begin{tikzpicture}[xscale=1.25,yscale=1.5]
    \node at (-8,7.5) {$\B^{\mathrm{id}}_{(2,1,1,0)}$};
    \node at (-8,7)     (E35) {$\tableau{3 \\ 2 \\ 1 & 1}$};
    \node at (-6.25,7.5) {$\B^{s_1}_{(2,1,1,0)}$};
    \node at (-6.25,7)     (D35) {$\tableau{3 \\ 2 \\ 1 & 1}$};
    \node at (-7,6)     (D14) {$\tableau{3 \\ 2 \\ 1 & 2}$};
    \draw[thick,blue  ,->](D35) -- (D14)   node[midway,above]{$1$};
    \node at (-4.5,7.5) {$\B^{s_3 s_1}_{(2,1,1,0)}$};
    \node at (-4.5,7)     (C35) {$\tableau{3 \\ 2 \\ 1 & 1}$};
    \node at (-5.25,6)     (C14) {$\tableau{3 \\ 2 \\ 1 & 2}$};
    \node at (-3.75,6)     (C54) {$\tableau{4 \\ 2 \\ 1 & 1}$};
    \node at (-4.5,5)     (C33) {$\tableau{4 \\ 2 \\ 1 & 2}$};
    \draw[thick,blue  ,->](C35) -- (C14)   node[midway,above]{$1$};
    \draw[thick,blue  ,->](C54) -- (C33)   node[midway,above]{$1$};
    \draw[thick,violet,->](C35) -- (C54)   node[midway,above]{$3$};
    \draw[thick,violet,->](C14) -- (C33)   node[midway,above]{$3$};
    \node at (-2,7.5) {$\B^{s_2 s_3 s_1}_{(2,1,1,0)}$};
    \node at (-2,7)     (B35) {$\tableau{3 \\ 2 \\ 1 & 1}$};
    \node at (-2.75,6)     (B14) {$\tableau{3 \\ 2 \\ 1 & 2}$};
    \node at (-1.25,6)     (B54) {$\tableau{4 \\ 2 \\ 1 & 1}$};
    \node at (-2.75,5)     (B13) {$\tableau{3 \\ 2 \\ 1 & 3}$};
    \node at (-2,5)     (B33) {$\tableau{4 \\ 2 \\ 1 & 2}$};
    \node at (-1.25,5)     (B53) {$\tableau{4 \\ 3 \\ 1 & 1}$};
    \node at (-2,4)     (B32) {$\tableau{4 \\ 2 \\ 1 & 3}$};
    \node at (-2,3)     (B31) {$\tableau{4 \\ 3 \\ 1 & 3}$};
    \draw[thick,blue  ,->](B35) -- (B14)   node[midway,above]{$1$};
    \draw[thick,blue  ,->](B54) -- (B33)   node[midway,above]{$1$};
    \draw[thick,purple,->](B14) -- (B13)   node[midway,left ]{$2$};
    \draw[thick,purple,->](B54) -- (B53)   node[midway,left ]{$2$};
    \draw[thick,purple,->](B33) -- (B32)   node[midway,left ]{$2$};
    \draw[thick,purple,->](B32) -- (B31)   node[midway,left ]{$2$};
    \draw[thick,violet,->](B35) -- (B54)   node[midway,above]{$3$};
    \draw[thick,violet,->](B14) -- (B33)   node[midway,above]{$3$};
    \node at (0.5,7.5) {$\B^{s_1 s_2 s_3 s_1}_{(2,1,1,0)}$};
    \node at (0.5,7)     (A35) {$\tableau{3 \\ 2 \\ 1 & 1}$};
    \node at (-0.25,6)     (A14) {$\tableau{3 \\ 2 \\ 1 & 2}$};
    \node at (1.25,6)     (A54) {$\tableau{4 \\ 2 \\ 1 & 1}$};
    \node at (-0.25,5)     (A13) {$\tableau{3 \\ 2 \\ 1 & 3}$};
    \node at (0.5,5)     (A33) {$\tableau{4 \\ 2 \\ 1 & 2}$};
    \node at (1.25,5)     (A53) {$\tableau{4 \\ 3 \\ 1 & 1}$};
    \node at (0.5,4)     (A32) {$\tableau{4 \\ 2 \\ 1 & 3}$};
    \node at (1,4)   (A42) {$\tableau{4 \\ 3 \\ 1 & 2}$};
    \node at (-0.25,3)     (A11) {$\tableau{4 \\ 3 \\ 2 & 2}$};
    \node at (0.5,3)     (A31) {$\tableau{4 \\ 3 \\ 1 & 3}$};
    \node at (-0.25,2)     (A10) {$\tableau{4 \\ 3 \\ 2 & 3}$};
    \draw[thick,blue  ,->](A35) -- (A14)   node[midway,above]{$1$};
    \draw[thick,blue  ,->](A54) -- (A33)   node[midway,above]{$1$};
    \draw[thick,blue  ,->](A53) -- (A42)   node[midway,right]{$1$};
    \draw[thick,blue  ,->](A42) -- (A11)   node[midway,above]{$1$};
    \draw[thick,blue  ,->](A31) -- (A10)   node[midway,above]{$1$};
    \draw[thick,purple,->](A14) -- (A13)   node[midway,left ]{$2$};
    \draw[thick,purple,->](A54) -- (A53)   node[midway,left ]{$2$};
    \draw[thick,purple,->](A33) -- (A32)   node[midway,left ]{$2$};
    \draw[thick,purple,->](A32) -- (A31)   node[midway,left ]{$2$};
    \draw[thick,purple,->](A11) -- (A10)   node[midway,left ]{$2$};
    \draw[thick,violet,->](A35) -- (A54)   node[midway,above]{$3$};
    \draw[thick,violet,->](A14) -- (A33)   node[midway,above]{$3$};
    \end{tikzpicture}
    \caption{\label{fig:TableauxDemazureCrystals}The Demazure crystals $\B^w_{(2,1,1,0)}$ for the permutations $w= \mathrm{id}, s_1 , s_3 s_1, s_2 s_3 s_1$, and $s_1 s_2 s_3 s_1$.}
  \end{center}
\end{figure}

For example, Fig.~\ref{fig:TableauxDemazureCrystals} constructs the Demazure crystal $\B^{4213}_{(2,1,1,0)}$ using the expression $4213 = s_1 s_2 s_3 s_1$ together with the left hand crystal in Fig.~\ref{fig:TableauxCrystals}. Beginning with the topmost tableau, traverse the single $f_1$ edge, followed by the two $f_3$ edges (taking the induced subgraph gives us an additional $f_1$ edge as well), followed by three vertical $f_2$ paths (the outer two of length $1$ and the middle of length $2$), and finally take two $f_1$ paths (and the one induced $f_2$ edge).

Just as Demazure crystals combinatorialize Demazure characters, the Demazure operator combinatorializes the isobaric divided difference operator.

\begin{proposition}\label{prop:Di}
  Let $X$ be a Demazure subset of a highest weight crystal $\B$.
  Then on the level of characters, we have
  \begin{equation}
      \ch(\Dem_i(X)) = \pi_i (\ch(X)) .
  \end{equation}
\end{proposition}

\begin{proof}
  For a subset $X \subseteq \B$ to be a Demazure crystal, we must have a decomposition
  \begin{eqnarray*}
   X & \cong & \B^{w^{(1)}}_{\lambda^{(1)}} \sqcup \cdots \sqcup \B^{w^{(k)}}_{\lambda^{(k)}} \\
   \ch(X) & = & \ch(\B^{w^{(1)}}_{\lambda^{(1)}}) + \cdots + \ch(\B^{w^{(k)}}_{\lambda^{(k)}}) 
   \end{eqnarray*}
  for some partitions $\lambda^{(1)},\ldots,\lambda^{(k)}$ and permutations $w^{(1)},\ldots,w^{(k)}$, where $k$ is the number of connected components of the crystal. The Demazure operator $\Dem_i$ acts on each connected component separately, and so factors through the decomposition. Thus it suffices to show for $\lambda$ a partition and $w$ a permutation, we have
  \[ \ch(\Dem_i(\B^w_{\lambda})) = \pi_i(\ch(\B^w_{\lambda})) . \]
  If $w \prec s_i \cdot w$ in Bruhat order, then by \eqref{e:BwL}, we have $\Dem_i(\B^w_{\lambda}) = \B_{s_i w}(\lambda)$, and so by \eqref{e:si}, we have
  \[ \ch(\Dem_i(\B^w_{\lambda})) = \ch(\B^{s_i w}_{\lambda}) = \pi_i(\ch(\B^w_{\lambda})). \]
  On the other hand, if $s_i \cdot w \prec w$ in Bruhat order, then $w$ has a reduced expression of the form $s_i s_{i_m} \cdots s_{i_1}$. Thus by \eqref{e:BwL}, we have $\B^w_{\lambda} = \Dem_i(\B^{s_i w}_{\lambda})$, and so  
  \[ \Dem_i(\B^w_{\lambda}) = \Dem_i(\Dem_i(\B^{s_i w}_{\lambda})) = \Dem_i(\B^{s_i w}_{\lambda}) = \B^w_{\lambda}. \]
  In this case, by symmetry of crystal strings, we also have $s_i \cdot \ch(\B^w_{\lambda}) = \ch(\B^w_{\lambda})$, and so $\pi_i(\ch(\B^w_{\lambda})) = \ch(\B^w_{\lambda})$. Thus in this case, we have
  \[ \ch(\Dem_i(\B^w_{\lambda})) = \ch(\B^w_{\lambda}) = \pi_i(\ch(\B^{w}_{\lambda})). \]
  The result follows.
\end{proof}

\subsection{Kohnert crystals}

Assaf \cite{Ass-KC} defined a crystal structure on diagrams that intertwines the crystal operators on tableaux via the injective, weight-reversing map from $\KD(\comp{a})$ to $\SSYT_n(\mathrm{sort}(\comp{a}))$ defined by Assaf and Searles \cite[Def~4.5]{AS18}. 

\begin{definition}\label{def:pair-diagram}\cite{Ass-KC}
  For $T$ a diagram and $1 \leq i <n$, we \newword{i-pair} the cells of $T$ in rows $i,i+1$ iteratively by $i$-pairing an unpaired cell in row $i+1$ with an unpaired cell in row $i$ weakly to its left whenever all cell in rows $i$ or $i+1$ lying between them are already $i$-paired.
\end{definition}

\begin{definition}\label{def:raise-diagram}\cite{Ass-KC}
  For $1 \leq i <n$, the \newword{Kohnert raising operator} $e_i$ acts on a diagram $T$ as follows:
  \begin{itemize}
  \item if all cells in row $i+1$ of $T$ are $i$-paired, then $e_i(T)=0$;
  \item otherwise, $e_i$ moves the rightmost unpaired cell in row $i+1$ to row $i$.
  \end{itemize}
\end{definition}

Given the asymmetry between raising and lowering operators for Demazure crystals, we define Kohnert crystals using the former. For example, Fig.~\ref{fig:KohnertCrystalNW} shows the Kohnert crystal structure on the set of Kohnert diagrams generated in Fig.~\ref{fig:KohnertPosetNW}.

\begin{definition}\label{def:lower-diagram}\cite{Ass-KC}
  For $1 \leq i <n$, the \newword{Kohnert lowering operator} $f_i$ acts on a diagram $T$ as follows:
  \begin{itemize}
  \item if all cells in row $i$ of $T$ are $i$-paired, then $f_i(T)=0$;
  \item otherwise, $e_i$ moves the leftmost unpaired cell in row $i$ to row $i+1$.
  \end{itemize}
\end{definition}

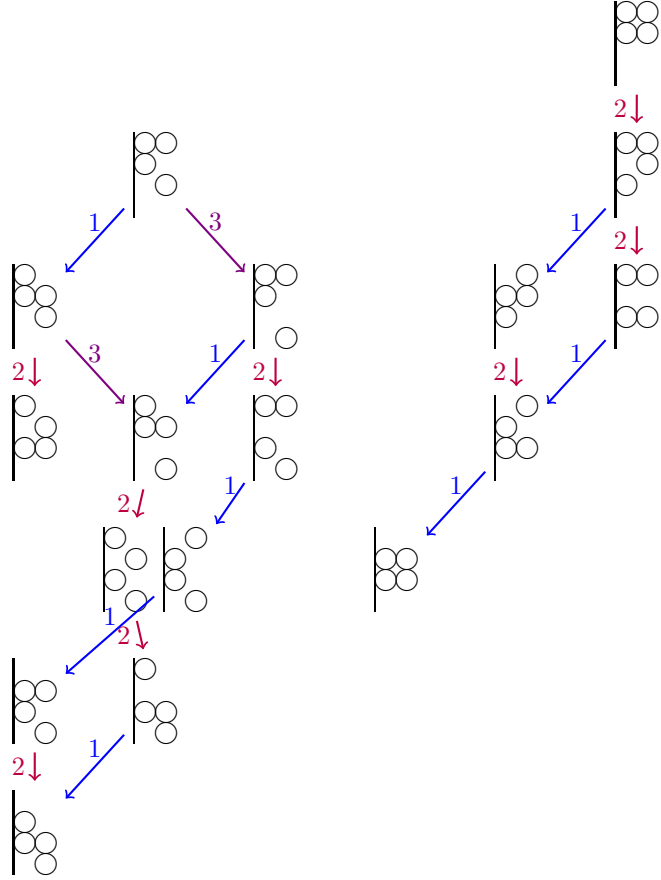
\begin{figure}[ht]
 \begin{center}
    \begin{tikzpicture}[xscale=1.6,yscale=1.75]
    \node at (2,5)     (A35) {$\cirtab{~ & ~ \\ ~ & \\ & ~ \\ & \\}$};
    \node at (1,4)     (A14) {$\cirtab{~ & \\ ~ & ~ \\ & ~ \\ & \\}$};
    \node at (3,4)     (A54) {$\cirtab{~ & ~ \\ ~ & \\ & \\ & ~ \\}$};
    \node at (1,3)     (A13) {$\cirtab{~ & \\ & ~ \\ ~ & ~ \\ & \\}$};
    \node at (2,3)     (A33) {$\cirtab{~ & \\ ~ & ~ \\ & \\ & ~ \\}$};
    \node at (3,3)     (A53) {$\cirtab{~ & ~ \\ & \\ ~ & \\ & ~ \\}$};
    \node at (1.75,2)     (A32) {$\cirtab{~ & \\ & ~ \\ ~ & \\ & ~ \\}$};
    \node at (2.25,2)     (A42) {$\cirtab{ & ~ \\ ~ & \\ ~ & \\ & ~ \\}$};
    \node at (1,1)     (A11) {$\cirtab{ & \\ ~ & ~ \\ ~ & \\ & ~ \\}$};
    \node at (2,1)     (A31) {$\cirtab{ ~ & \\ & \\ ~ & ~ \\ & ~ \\}$};
    \node at (1,0)     (A10) {$\cirtab{ & \\ ~ & \\ ~ & ~ \\ & ~ \\}$};
    \draw[thick,blue  ,->](A35) -- (A14)   node[midway,above]{$1$};
    \draw[thick,blue  ,->](A54) -- (A33)   node[midway,above]{$1$};
    \draw[thick,blue  ,->](A53) -- (A42)   node[midway,above]{$1$};
    \draw[thick,blue  ,->](A42) -- (A11)   node[midway,above]{$1$};
    \draw[thick,blue  ,->](A31) -- (A10)   node[midway,above]{$1$};
    \draw[thick,purple,->](A14) -- (A13)   node[midway,left ]{$2$};
    \draw[thick,purple,->](A54) -- (A53)   node[midway,left ]{$2$};
    \draw[thick,purple,->](A33) -- (A32)   node[midway,left ]{$2$};
    \draw[thick,purple,->](A32) -- (A31)   node[midway,left ]{$2$};
    \draw[thick,purple,->](A11) -- (A10)   node[midway,left ]{$2$};
    \draw[thick,violet,->](A35) -- (A54)   node[midway,above]{$3$};
    \draw[thick,violet,->](A14) -- (A33)   node[midway,above]{$3$};
    \node at (6,6)     (B75) {$\cirtab{ ~ & ~ \\ ~ & ~ \\ & \\ & \\}$};
    \node at (6,5)     (B74) {$\cirtab{ ~ & ~ \\ & ~ \\ ~ & \\ & \\}$};
    \node at (6,4)     (B73) {$\cirtab{ ~ & ~ \\ & \\ ~ & ~ \\ & \\}$};
    \node at (5,4)     (B63) {$\cirtab{ & ~ \\ ~ & ~ \\ ~ & \\ & \\}$};
    \node at (5,3)     (B62) {$\cirtab{ & ~ \\ ~ & \\ ~ & ~ \\ & \\}$};
    \node at (4,2)     (B51) {$\cirtab{ & \\ ~ & ~ \\ ~ & ~ \\ & \\}$};
    \draw[thick,blue  ,->](B74) -- (B63)   node[midway,above]{$1$};
    \draw[thick,blue  ,->](B73) -- (B62)   node[midway,above]{$1$};
    \draw[thick,blue  ,->](B62) -- (B51)   node[midway,above]{$1$};
    \draw[thick,purple,->](B75) -- (B74)   node[midway,left ]{$2$};
    \draw[thick,purple,->](B74) -- (B73)   node[midway,left ]{$2$};
    \draw[thick,purple,->](B63) -- (B62)   node[midway,left ]{$2$};
    \end{tikzpicture}
    \caption{\label{fig:KohnertCrystalNW}The Kohnert crystal on Kohnert diagrams for the bottom diagram, which is northwest.}
  \end{center}
\end{figure}

Assaf \cite[Thm~4.1.1]{Ass-KC} proves the Kohnert raising operators are closed within the set of Kohnert diagrams, provided the initial diagram is northwest. 

\begin{theorem}\cite{Ass-KC}\label{thm:KC-raise}
  For $D$ a northwest diagram, if $T\in\KD(D)$ and $e_i(T)\neq 0$, then $e_i(T)\in\KD(D)$.
\end{theorem}

We remark the analog of Theorem~\ref{thm:KC-raise} for Kohnert lowering operators is false in general, though true under certain circumstances considered in Lemma~\ref{lem:left-justified} below. However, for $D$ northwest, Theorem~\ref{thm:KC-raise} makes the following well-defined.

\begin{definition}\cite{Ass-KC}\label{def:KC}
  For $D$ a northwest diagram, the \newword{Kohnert crystal on $D$} is the set $\KD(D)$ together with the Kohnert raising operators.
\end{definition}

For example Fig.~\ref{fig:KohnertCrystalNW} shows the Kohnert crystal for the bottom most diagram. Notice there are two connected components, and the leftmost corresponds precisely with the rightmost Demazure crystal in Fig.~\ref{fig:TableauxDemazureCrystals}. One can easily verify the rightmost component is also a Demazure crystal, and Theorem~5.3.4 of \cite{Ass-KC} proves the Kohnert crystal for northwest diagrams is always a union of Demazure crystals.

\begin{theorem}[\cite{Ass-KC}]\label{thm:5.3.4}
  For $D$ a northwest diagram, the Kohnert crystal on $\KD(D)$ is a disjoint union of Demazure crystals. That is, there exist partitions $\lambda^{(1)},\ldots, \lambda^{(m)}$ and permutations $w^{(1)},\ldots,w^{(m)}$ such that
  \[ \KD(D) \cong \B^{w^{(1)}}_{\lambda^{(1)}} \sqcup \cdots \sqcup \B^{w^{(m)}}_{\lambda^{(m)}}. \]
\end{theorem}

To establish Theorem~\ref{thm:recur}(M3) and prove our main result, by Proposition~\ref{prop:Di}, it is enough to show that for $D$ northwest with row $r$ a subset of row $r+1$, we have 
\begin{equation}\label{e:Dem_si}
    \KD(D) = \Dem_r(\KD(s_r D)) .
\end{equation} 
By Lemma~\ref{lem:subset}, we have $\KD(s_r D) \subseteq \KD(D)$. To begin to relate $\KD(D)$ with $\Dem_r(\KD(s_r D))$, we have the following.

\begin{lemma}\label{lem:hwt}
    Let $D$ be a diagram, and $r$ a row index such that row $r$ is contained in row $r+1$. Given $T\in\KD(D)$, there exists $S\in\KD(s_r D)$ such that 
  \begin{enumerate}
      \item[(1)] $S$ and $T$ agree in all rows $t \neq r,r+1$;
      \item[(2)] if $S,T$ differ in some column $c$ at row $r$ or $r+1$, then $T$ has a cell only in row $r+1$ and $S$ has a cell only in row $r$.
      \item[(3)] if, at rows $r,r+1$, $S,T$ differ in column $c$ and agree in some column $b<c$, and if, in column $b$, row $r$ has a cell, so does row $r+1$.
  \end{enumerate}
\end{lemma}

\begin{proof}
  We proceed by induction on the number of Kohnert moves used to obtain $T$ from $D$. For the base case, we observe $s_r D$ is obtained by applying raising operators $e_r$ to $D$ by the proof of Lemma~\ref{lem:subset}, and so for $T=D$ we may take $S=s_r D$. Thus suppose for some $T\in\KD(D)$, we have constructed $S\in\KD(s_r D)$ satisfying conditions (1)-(3). Suppose $T'$ is obtained by a Kohnert move on $T$, say moving a cell $x$ which lies in column $c$. We will construct $S'$ by a (possibly trivial) sequence of Kohnert moves on $S$ such that conditions (1)-(3) are satisfied for $S'$ and $T'$. 
  
  \textbf{Case A: $x$ not in rows $r,r+1$.}
  \begin{itemize}
  \item \textbf{Case A1:} $T'$ is obtained by moving $x$ between rows that lie either strictly above row $r$ or strictly below row $r+1$. 
  
  By condition (1), $S$ and $T$ agree in all rows $t \neq r, r+1$, so we can perform the same Kohnert move on $S$ to obtain $S'$. Since this Kohnert move does not move any cells in rows $r$ or $r+1$, conditions (1)-(3) hold by induction.
  
  \item \textbf{Case A2:} $T'$ is obtained by moving $x$ from row $s > r +1$ to row $q < r$.
  
  In order to perform this Kohnert move, $T$ must have cells in rows $r$ and $r+1$ of column $c$. By condition (2) $S$ and $T$ must then agree in rows $r$ and $r+1$ of column $c$. Thus by condition (1) we can perform the same Kohnert move on $S$ to obtain $S'$, and conditions (1)-(3) once again hold by induction since this Kohnert move does affect any cells in rows $r$ or $r+1$.
  
  \item \textbf{Case A3:} $T'$ is obtained by moving $x$ from row $s > r+1$ to row $r+1$.
  
  Since $x$ moves into row $r+1$ from $T$ to $T'$, condition (2) implies that $S$ and $T$ agree in column $c$; otherwise, $T$ would have a cell in row $r+1$, column $c$, which would block $x$ from moving into row $r+1$. Then by condition (1) we can again perform the same Kohnert move on $S$ to obtain $S'$. In particular, $S'$ and $T'$ agree in column $c$, so conditions (1) and (2) immediately follow by induction. Since $S'$ and $T'$ both have a cell in row $r+1$, column $c$, the implication of condition (3) is satisfied as well. 
  
  \item \textbf{Case A4:} $T'$ is obtained by moving $x$ from row $s > r+1$ to row $r$.
  
  Here, $S$ and $T$ may not agree in column $c$ by condition (2). If this is the case, then in rows $r$ and $r+1$ of column $c$, $T$ has a cell only in row $r+1$ while $S$ has a cell only in row $r$. However, after moving $x$, $T'$ has cells in rows $r$ and $r+1$ of column $c$. Thus in $S$, we can move $x$ up to row $r+1$ so that $S'$ and $T'$ both have cells in rows $r$ and $r+1$ of column $c$. 
  
  If, on the other hand, $S$ and $T$ do agree in column $c$, then we can perform the same Kohnert move on $S$ which we performed on $T$ to obtain the same result: $S'$ and $T'$ both have cells in rows $r$ and $r+1$ of column $c$, and therefore the columns are identical in either case by condition (1). Conditions (1) and (2) again follow by induction, and condition (3) holds because $S'$ and $T'$ both have cells in rows $r$ and $r+1$ of column $c$.
  \end{itemize}

  \textbf{Case B: $x$ in row $r+1$.}
 \begin{itemize}
 \item \textbf{Case B1:} $S$ and $T$ agree in column $c$.
 
  In this case we can perform the same Kohnert move on $S$ to obtain $S'$. Conditions (1) and (2) then hold by induction. For condition (3), notice that since $x$ is able to move out of row $r+1$ from $T$ to $T'$, there can be no columns right of $c$ in which $S$ and $T$ differ; otherwise by condition (2) there would be a cell in row $r+1$ of $T$ which blocks $x$ from moving. Thus condition (3) holds vacuously for every column weakly right of $c$, and holds by induction for every column left of $c$.
 
 \item \textbf{Case B2:} $S$ and $T$ differ in column $c$. 
 
 By condition (2), $T$ has a cell only in row $r+1$, while $S$ only has a cell in row $r$, in column $c$. In this case $x$ lies in row $r$ in $T'$, and we define $S' = S$ so that $S$ and $T$ agree in column $c$. Conditions (1) and (2) hold by induction. Condition (3) also holds by induction because, as we noted above, there can be no columns right of $c$ in which $S$ and $T$ differ.
 \end{itemize}

  \textbf{Case C: $x$ in row $r$. } By condition (2), $S$ and $T$ must agree in column $c$.
\begin{itemize}
\item \textbf{Case C1:} $S$ and $T$ agree in all columns right of $c$.

In this case, we can perform the same Kohnert move on $S$ to obtain $S'$ so that $S'$ and $T'$ agree in column $c$. Then conditions (1) and (2) hold by induction. By assumption column $c$ is right of all columns in which $S$ and $T$ differ, so condition (3) also holds by induction.

\item \textbf{Case C2:} $S$ and $T$ differ in some column right of $c$.

Here $S$ and $T$ both have cells in row $r$, column $c$ and there exists some column $d > c$ in which $S$ and $T$ differ. By condition (3), $S$ and $T$ must both have cells in rows $r$ and $r+1$ of column $c$. Furthermore $T$ has no cells in row $r$ right of column $c$ --- otherwise $x$ is blocked from moving --- and this implies that there is at most one cell in rows $r$ and $r+1$ per column right of $c$ in $S$. Indeed, for any column $d > c$ either
\begin{enumerate}
\item[(i)] $S$ and $T$ differ in rows $r$ and $r+1$ of column $d$, and condition (2) implies that $S$ has a single cell in row $r$, column $d$, or 
\item[(ii)] $S$ and $T$ agree in rows $r$ and $r+1$ of column $d$, in which case $S$ has no cells in row $r$, column $d$, since $T$ does not either. 
\end{enumerate}

We define $S'$ as follows: first, move every cell in row $r+1$ of $S$ which lies in a column strictly right of $c$ into row $r$ via a sequence of Kohnert moves to obtain an intermediate diagram $\tilde{S}$. Let $y$ be the cell in row $r+1$, column $c$ of $S$ (which is in the same position in $\tilde{S}$). Then perform a Kohnert move on $\tilde{S}$ by moving $y$ to obtain $S'$. By construction, $y$ lands in the same position in $S'$ as $x$ lands in $T'$. 

Condition (1) holds as the only new cell introduced outside of rows $r$ and $r+1$ lies in the same position in $S'$ and $T'$. Condition (2) holds in all columns left of $c$ by induction, and we now show it holds in all columns weakly right of $c$. In rows $r$ and $r+1$ of column $c$, $T'$ has a cell only in row $r+1$ and $S'$ has a cell only in row $r$, so condition (2) holds. For any column $d > c$, we have two possibilities:
\begin{enumerate}
\item[(i)] $S$ and $T$ differ in rows $r$ and $r+1$. Then $S$ has a cell in row $r$, column $d$, and this cell is in the same place in $S'$. Since $T$ and $T'$ agree right of column $c$, condition (2) holds.
\item[(ii)] $S$ and $T$ agree in rows $r$ and $r+1$. Without loss of generality column $d$ is nonempty, so that $S$ and $T$ have a cell only in row $r+1$ of column $d$. Then in $S'$ this cell is moved to row $r$, column $d$, while it stays fixed in $T'$. Therefore $S'$ and $T'$ differ in column $d$ in rows $r$ and $r+1$; $T'$ has a cell only in row $r+1$ and $S'$ has a cell only in row $r$, so condition (2) is satisfied.
\end{enumerate}
We have just shown that $S'$ and $T'$ differ in rows $r$ and $r+1$ in all columns right of $c$ in accordance with condition (2). Since we assumed that $S$ and $T$ differ in some column right of $c$, condition (3) holds by induction.
\end{itemize}
Thus $S^\prime$ is constructed so that (1)-(3) hold. The result follows by induction.
\end{proof}

We now establish that $\KD(s_r D)$ and $\KD(D)$ have the same number of connected components, each with the same highest weight.

\begin{theorem}\label{thm:hwt}
 Let $D$ be a northwest diagram and $r$ a row index such that row $r$ is contained in row $r+1$. By Theorem~\ref{thm:5.3.4}, suppose the Kohnert crystals decompose into Demazure crystals (with minimal length permutations) as 
 \[ \KD(s_r D) \cong \bigsqcup_{i=1}^{s} \B^{u^{(i)}}_{\mu^{(i)}} 
   \hspace{1em} \text{and} \hspace{1em}
   \KD(D) \cong \bigsqcup_{i=1}^{t} \B^{v^{(i)}}_{\nu^{(i)}} .
 \]
Then $s=t$, and for all $i$, we have $\mu^{(i)} = \nu^{(i)}$ and $u^{(i)} \preceq v^{(i)}$ in Bruhat order. 
\end{theorem}

\begin{proof}
  Each connected Demazure crystal $\B^w_{\lambda}$ has a unique highest weight element, say $T$, and $T$ satisfies $\wt(T)=\lambda$. Thus there exist elements $T_1,\ldots,T_s\in\KD(s_r D)$ such that $e_i(T_j)=0$ for all $i,j$, and $\wt(T_j) = \mu^{(j)}$. By Lemma~\ref{lem:subset}, we have $T_j\in\KD(D)$, and since the crystal operators are independent of the ambient diagram, each $T_j$ is also a highest weight for $\KD(D)$.

  Conversely, let $T$ be a highest weight element of $\KD(D)$, and let $S\in \KD(s_r D)$ be as in Lemma~\ref{lem:hwt}. Suppose there exists a column $c$ where $T$ and $S$ disagree. By condition (2), $T$ has a cell $x$ in row $r+1$ but not in row $r$ of this column. Any column $b<c$ that has a cell $y$ in row $r$ agrees with $S$ by condition (2), hence by condition (3) also has a cell in row $r+1$ to which $y$ must be $r$-paired, and so $x$ cannot be $r$-paired to $y$. This contradicts the premise that $e_r(T)=0$ so we conclude $T$ and $S$ are identical, that is $T\in \KD(s_r D)$. In particular, $\KD(s_r D)$ contains the highest weight diagrams of $\KD(D)$.
  
  That $u^{(i)}\preceq w^{(i)}$ follows because $\B^{u^{(i)}}_{\lambda^{(i)}}\subseteq \B^{w^{(i)}}_{\lambda^{(i)}}$ by Lemma~\ref{lem:subset}.
\end{proof}

We reformulate Theorem~\ref{thm:hwt} in terms of the Demazure operators as follows.

\begin{corollary}\label{cor:nested}
  For $D$ northwest with row $r$ a subset of row $r+1$, we have $f_r^k(s_r D) = D$, where $k$ is the maximum integer $i$ such that $f_r^i(s_r D) \neq 0$. In particular, we have nested crystals 
  \[ \KD(s_r D) \subseteq \KD(D) \subseteq \Dem_r(\KD(s_r D)) , \]
  and each has the same number of connected components.
\end{corollary}

In order to prove the latter containment is an equality, thus establishing \eqref{e:Dem_si}, we will show $\Dem_r(\KD(D)) = \KD(D)$ whenever $D$ is northwest with row $r$ a subset of row $r+1$. This amounts to showing the lowering analog of \cite[Thm 4.1.1]{Ass-KC}.

\begin{lemma}\label{lem:left-justified}
 Let $D$ be a left-justified diagram such that row $r \subseteq$ row $r+1$. Then $\Dem_r(\KD(D)) = \KD(D)$. In particular, for any $T \in \KD(D)$ such that $f_r(T) \neq 0$, we have $f_r(T) \in \KD(D)$.
\end{lemma}

\begin{proof}
  Let $\comp{a} = \wt(D)$, and let $w$ be the minimal length permutation that sorts $\comp{a}$ to a partition, say $\lambda$. Then $\KD(D) \cong \B^w_{\lambda}$. The row containment condition on $D$ translates to $\comp{a}$ as $a_r \leq a_{r+1}$, and so $s_r \cdot w \prec w$ in Bruhat order. In particular, $w$ has a reduced expression of the form $s_r s_{i_k} \cdots s_{i_1}$. Then we have
  \[ \KD(D) \cong \B^w_{\lambda} \cong \Dem_r \Dem_{i_k} \cdots \Dem_{i_1} (\{u_{\lambda}\}). \]
  The result now follows from the observation $\Dem_r^2(X) = \Dem_r(X)$ for any $X$. 
\end{proof}

To establish this for $D$ northwest, we leverage labelings of Kohnert diagrams introduced in \cite{AS18} and generalized in \cite{Ass-KC}.

%
\section{Labelings of diagrams}
%
\label{sec:labeling}

Unlike tableaux, Kohnert diagrams do not, by default, have labels their cells. However, in this section we consider a canonical labeling of the cells of a diagram that allows one to determine if a given diagram belongs to a set $\KD(D)$ for a particular northwest diagram $D$. Thus it will be essential for showing $f_r(T)\in\KD(D)$ for $T\in\KD(D)$ when $D$ is northwest with row $r$ a subset of row $r+1$.

\subsection{Kohnert tableaux}
\label{sec:labeling-kohnert}

Assaf and Searles \cite[Def~2.5]{AS18} give an algorithm by which the cells of a diagram $T$ are labeled with respect to a weak composition $\comp{a}$ in order to determine whether or not $T\in\KD(\D(\comp{a}))$.

\begin{definition}[\cite{AS18}]
  For a weak composition $\comp{a}$ and a diagram $T\in\KD(\D(\comp{a}))$, the \newword{Kohnert labeling of $T$ with respect to $\comp{a}$}, denoted by $\Label_{\comp{a}}(T)$, assigns labels to cells of $T$ as follows. Assuming all columns right of column $c$ have been labeled, assign labels $\{i \mid a_i \geq c\}$ to cells of column $c$ from top to bottom by choosing the smallest label $i$ such that the $i$ in column $c+1$, if it exists, is weakly higher.
\label{def:labelling_algorithm}
\end{definition}

For example, the columns of the diagram in Figure~\ref{fig:label} are labeled with respect to $(3,0,5,1,4)$, from right to left, with entries $\{3\}$, $\{3,5\}$, $\{1,3,5\}$, $\{1,3,5\}$, $\{1,3,4,5\}$.

\begin{figure}[ht]
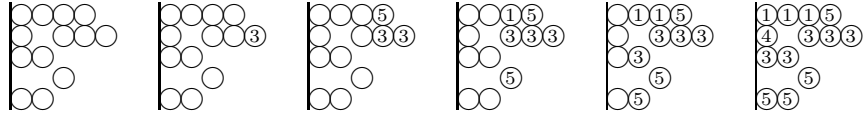

    \begin{displaymath}
    \arraycolsep=\cellsize
    \begin{array}{cccccc}
      \cirtab{ ~ & ~ & ~ & ~ \\ ~ & & ~ & ~ & ~ \\ ~ & ~ & \\ & & ~ \\ ~ & ~ } & 
      \cirtab{ ~ & ~ & ~ & ~ \\ ~ & & ~ & ~ & 3 \\ ~ & ~ & \\ & & ~ \\ ~ & ~ } & 
      \cirtab{ ~ & ~ & ~ & 5 \\ ~ & & ~ & 3 & 3 \\ ~ & ~ & \\ & & ~ \\ ~ & ~ } & 
      \cirtab{ ~ & ~ & 1 & 5 \\ ~ & & 3 & 3 & 3 \\ ~ & ~ & \\ & & 5 \\ ~ & ~ } & 
      \cirtab{ ~ & 1 & 1 & 5 \\ ~ & & 3 & 3 & 3 \\ ~ & 3 & \\ & & 5 \\ ~ & 5 } & 
      \cirtab{ 1 & 1 & 1 & 5 \\ 4 & & 3 & 3 & 3 \\ 3 & 3 & \\ & & 5 \\ 5 & 5 } 
    \end{array}
    \end{displaymath}
\caption{\label{fig:label}The labeling algorithm applied to a Kohnert diagram for $(3,0,5,1,4)$.}
\end{figure}

To handle the ambiguity where a given diagram is a Kohnert diagram for multiple weak compositions, Assaf and Searles \cite[Def~2.3]{AS18} define a \emph{Kohnert tableaux} to be labelings that carry the information of the weak composition as well.

\begin{definition}[\cite{AS18}]
  For a weak composition $\comp{a}$ of length $n$, a \newword{Kohnert tableau of content $\comp{a}$} is a diagram labeled with $1^{a_1}, 2^{a_2}, \ldots, n^{a_n}$ satisfying
  \begin{enumerate}[label=(\roman*)]
  \item \label{i:label} there is exactly one $i$ in each column from $1$ through $a_i$;
  \item \label{i:flag} each entry in row $i$ is at least $i$;
  \item \label{i:descend} the cells with entry $i$ weakly ascend from left to right;
  \item \label{i:invert} if $i<j$ appear in a column with $i$ below $j$, then there is an $i$ in the column immediately to the right of and strictly below $j$.   
  \end{enumerate}
  \label{def:kohnert-tableaux}
\end{definition}

For example, Figure~\ref{fig:KT} shows the $11$ Kohnert tableaux of content $(0,1,2,1)$.

\begin{figure}[ht]
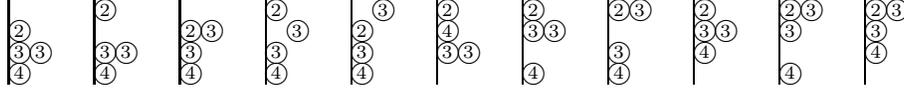

    \begin{displaymath}
    \arraycolsep=\cellsize
    \begin{array}{ccccccccccc}
      \cirtab{ & \\ 2 & \\ 3 & 3 \\ 4 } & 
      \cirtab{ 2 & \\ & \\ 3 & 3 \\ 4 } & 
      \cirtab{ & \\ 2 & 3 \\ 3 & \\ 4 } & 
      \cirtab{ 2 & \\ & 3 \\ 3 & \\ 4 } & 
      \cirtab{ & 3 \\ 2 & \\ 3 & \\ 4 } & 
      \cirtab{ 2 & \\ 4 & \\ 3 & 3 \\ } &
      \cirtab{ 2 & \\ 3 & 3 \\ & \\ 4 } & 
      \cirtab{ 2 & 3 \\ & \\ 3 & \\ 4 } & 
      \cirtab{ 2 & \\ 3 & 3 \\ 4 & \\ } & 
      \cirtab{ 2 & 3 \\ 3 & \\ & \\ 4 } & 
      \cirtab{ 2 & 3 \\ 3 & \\ 4 & \\ }  
    \end{array}
    \end{displaymath}
\caption{\label{fig:KT}The Kohnert tableaux of content $(0,1,2,1)$.}
\end{figure}

Assaf and Searles \cite[Thm~2.8]{AS18} prove for any diagram $T$ and weak composition $\comp{a}$ for which $T$ and $\D(\comp{a})$ have the same column weight, $T\in\KD(\D(\comp{a}))$ if and only if $\Label_{\comp{a}}(T)$ is a Kohnert tableau. Thus the labeling algorithm provides the bijection between Kohnert tableaux  of content $\comp{a}$ and Kohnert diagrams for $\comp{a}$.

In fact, the algorithm ensures conditions \ref{i:label}, \ref{i:descend}, and \ref{i:invert} always hold, so we can restate this result as follows.

\begin{proposition}[\cite{AS18}]
  Given a diagram $T$ and weak composition $\comp{a}$ for which $T$ and $\D(\comp{a})$ have the same column weight, $T\in\KD(\D(\comp{a}))$ if and only if each entry in row $i$ of $\mathcal{L}_{\comp{a}}(T)$ is at least $i$.
  \label{prop:flagged}
\end{proposition}

To state the analogous definitions for northwest diagrams, we must use \emph{rectification} \cite[Def~4.2.4]{Ass-KC}. Rectification maps any diagram to one that is a Kohnert diagram for a composition. Essentially, rectification is the result of transposing a diagram, applying all possible Kohnert lowering operators (in any order), and then transposing back. 

\begin{definition}[\cite{Ass-KC}]
  For $T$ a diagram and $i \geq 1$ an integer, we \newword{column $i$-pair} cells of $T$ in columns $i,i+1$ iteratively by column $i$-pairing an unpair cell in column $i+1$ with an unpaired cell in column $i$ weakly above it whenever all cells in columns $i$ and $i+1$ that lie strictly between them are already column $i$-paired.
  \label{def:KD-cpair}
\end{definition}

It follows from the column pairing rule and \cite[Lemma~2.2]{AS18} that a diagram is a Kohnert diagram for some weak composition if and only if for all $i$, all cells in column $i+1$ are column $i$-paired.

\begin{definition}[\cite{Ass-KC}]
  For $T$ a diagram and $i \geq 1$ an integer, the \newword{rectification operator} $\Rect_i$ acts on $T$ as follows: 
  \begin{itemize}
      \item if all cells in column $i+1$ of $T$ are column $i$-paired, then $\Rect_i(T)=T$;
      \item otherwise, $\Rect_i$ moves the lowest unpaired cell in column $i+1$ to column $i$.
  \end{itemize}
  \label{def:KD-rect}
\end{definition}

Just as applying any terminal sequence of raising operators results in a unique highest weight element on a crystal, Assaf \cite[Lemma~4.3.3]{Ass-KC} proves the following is well-defined.

\begin{definition}[\cite{Ass-KC}]
  Given a diagram $T$, the \newword{rectification of $T$}, denoted by $\Rect(T)$, is the unique  diagram obtained by applying any terminal sequence of rectification operators to $T$.
  \label{def:embed-KT}
\end{definition}

Rectification is a powerful tool for studying Kohnert crystals precisely because it intertwines the crystal operators \cite[Thm~4.2.5]{Ass-KC}. Moreover, it facilitates a generalization of the labeling algorithm and Kohnert tableaux.

Given diagrams $T,D$ of the same column weight, following \cite[Def~5.1.8]{Ass-KC}, we label $T$ with respect to $D$ by a greedy algorithm that uses rectification to follow the Assaf--Searles labeling algorithm. For this, given a diagram $T$ and a column $c$, partition $T$ at column $c$ by $T_{\leq c} \sqcup T_{>c}$, where the former contains cells in columns weakly left of $c$, and the latter contains  cells in columns strictly right of $c$. 

\begin{definition}[\cite{Ass-KC}]
  For diagrams $T,D$ of the same column weight, construct the \newword{Kohnert labeling of $T$ with respect to $D$}, denoted by $\Label_{D}(T)$, as follows. Once all columns of $T$ right of column $c$ have been labeled, set $\displaystyle T' = T_{\leq c} \ \sqcup \ \Rect(T_{>c})$, where the leftmost occupied column of the latter is $c+1$. Bijectively assign labels
  \[\{r \mid D \text{ has a cell in column $c$, row $r$ } \}\]
  to cells in column $c$ of $T$ from smallest to largest by assigning label $r$ to the lowest unlabeled cell $x$ such that if there exists a cell $z$ in column $c+1$ of $\Rect(T_{>c})$ with label $r$, then $x$ lies weakly above $z$.
  \label{def:labeling}
\end{definition}

Generalizing the characterization for left-justified diagrams to northwest diagrams, combining Theorems~5.2.2 and 5.2.4 from \cite{Ass-KC} proves the following.

\begin{theorem}[\cite{Ass-KC}]
  For $D$ a northwest diagram, $\Label_{D}$ is well-defined and flagged for $T$ if and only if $T\in\KD(D)$.
  \label{thm:nw-KT}
\end{theorem}

\subsection{Closure under lowering operators}
\label{sec:labeling-lower}

By Theorem~\ref{thm:nw-KT}, we can prove our main result by showing $\Label_{D}$ is well-defined and flagged for $f_r(T)$ whenever $T\in\KD(D)$ for $D$ northwest with row $r$ a subset of row $r+1$. For this we focus on the labels $r,r+1$.

\begin{lemma}\label{lem:labeling-restriction}
Let $D$ be a northwest diagram, and let $r$ be a row index such that for every cell in row $r$ there is a cell directly below it row $r+1$. Let $T$ be a diagram with the same column weights as $D$. Then in the Kohnert labeling $\mathcal{L}_D(T)$ of $T$ with respect to $D$, the label $r$ always appears above the label $r+1$ within columns.
\end{lemma}

\begin{proof}
We proceed by induction on the columns of $T$. Let $c$ be the largest index for which the original diagram $D$ has a cell in row $r$, column $c$. Then in column $c$ of $\mathcal{L}_D(T)$, the label $r$ must appear above the label $r+1$ by definition, as there are no cells labeled $r$ right of column $c$. 

Since $D$ is northwest, every column left of $c$ in $D$ either has no cells in row $r$ or row $r+1$, or has cells in both rows $r$ and $r+1$. Thus for each column $b < c$ in $\mathcal{L}_D(T)$, either column $b$ will contain both labels $r$ and $r+1$, or it will contain neither label. Assume that, if column $b$ does contain both labels, then the label $r$ appears above the label $r+1$.

Consider column $b-1$ and assume it contains both labels $r$ and $r+1$; otherwise, the result follows trivially. If column $b$ does not contain the labels $r$ and $r+1$, then in column $b-1$, the candidate cells for the labels $r$ and $r+1$ are the same. Thus since we assign labels from smallest to largest, the label $r$ will be placed above the label $r+1$ in column $b-1$. 

If, on the other hand, column $b$ does contain these labels, then the label $r$ appears above the label $r+1$. In column $b-1$, the label $r$ is placed in the highest available cell which is weakly below the cell with label $r$ in column $b$, and then the label $r+1$ is placed in the analogous manner. For the label $r+1$ to end up above the label $r$ in column $b-1$ would mean that there was a cell in column $b-1$ weakly below the cell labeled $r+1$ in column $b$ which was not weakly below the cell labeled $r$ in column $b$. However, this would force the cell labeled $r+1$ in column $b$ to lie above the cell labeled $r$, contradicting the assumption.
\end{proof}

Lemma~\ref{lem:labeling-restriction} shows cells labeled $r$ are bounded by those labeled $r+1$ in $\Label_D(T)$. We next show this persists under rectification. Thus Kohnert tableaux condition \ref{i:invert} on the rectification will establish the flagged condition on $\Label_D(T)$.

\begin{lemma}\label{lem:row_containment}
 Let $D$ be a northwest diagram, and let $r$ be a row index such that for every cell in row $r$ there is a cell directly below it row $r+1$. Let $T$ be a diagram with the same column weights as $D$. Then $\Rect(\mathcal{L}_D(T))$ has at least as many cells labeled $r+1$ as cells labeled $r$.
\end{lemma}

\begin{proof}
Say that a column $c$ of $T$ has property $(*)$ if, whenever column $c$ contains a cell with label $r$, then column $c$ also contains a cell with label $r+1$ weakly below the cell with label $r$. We now proceed by induction on the columns of $T$, showing that after we label rectify column $c$, both columns $c$ and $c+1$ satisfy $(*)$.

First observe that in the original Kohnert labeling $\mathcal{L}_D(T)$ of $T$ with respect to $D$, the result certainly holds since $D$ has more cells in row $r+1$ than in row $r$. Now assume we have label rectified up to column $c+1$, so that $(*)$ holds in all columns weakly right of $c+1$. In column $c$, we can have that neither $r$ not $r+1$ appears as a label (C0), that only $r+1$ appears as a label (C1), or that both $r$ and $r+1$ appear as labels (C2). We have the same three possibilities in column $c+1$; designate these cases as (D0), (D1), and (D2), respectively. \\

\textbf{Case (C0).}

\begin{itemize}
\item \textbf{Subcase (D0).} Nothing to check; never gain the label $r$ in either column.
\item \textbf{Subcase (D1).} Nothing to check; never gain the label $r$ in either column.
\item \textbf{Subcase (D2).} In column $c$, $(*)$ can fail if the cell labeled $r$ in column $c+1$ moves into column $c$, while the cell labeled $r+1$ in column $c+1$ does not. In column $c+1$, $(*)$ may fail if the cell labeled $r$ is not label paired, while the cell labeled $r+1$ is label paired with a cell in column $c$. By \cite[Lemma 5.2.3]{Ass-KC}, these cases are in fact equivalent. Furthermore note that in this case, the original diagram $D$ must have cells in rows $r$ and $r+1$ of column $c+1$, and have no cells in rows $r$ or $r+1$ of column $c$. This implies that $D$ cannot have cells anywhere in column $c$ below row $r$, since $D$ is northwest. As a result, the labels in column $c$ of $T$ are all $<r$.

Now let $x_0$ denote the cell in column $c+1$ labeled $r$ and let $x_1$ denote the cell in column $c+1$ labeled $r+1$. Assume $x_0$ is unpaired, and assume $x_1$ is label paired with some cell in column $c$. As noted above, the label on this latter cell is $<r$. 

Then by the first label rectifying ``rule," we find a cell $y$ in column $c$ which is label paired to a cell $z$ in column $c+1$ such that $\mathcal{L}(y) \le \mathcal{L}(x_0) < \mathcal{L}(z)$ and such that $\mathcal{L}(z)$ is maximal. (This is possible since we found such a pair in the previous paragraph.) We then swap the labels on $x_0$ and $z$ as we push $x_0$ into column $c$. As a result, column $c$ gains a cell with label $\mathcal{L}'(x_0) = \mathcal{L}(z) \ge r+1$, so $(*)$ is satisfied. In column $c+1$, $x_0$ has moved out and $z$ inherits $\mathcal{L}(y) < r$, so the column has no cell labeled $r$ and $(*)$ again holds.
\end{itemize}

\textbf{Case (C1).}

\begin{itemize}
\item \textbf{Subcase (D0).} Nothing to check; never gain the label $r$ in either column.
\item \textbf{Subcase (D1).} Nothing to check; never gain the label $r$ in either column.
\item \textbf{Subcase (D2).} Not possible, since $D$ is northwest.
\end{itemize}

\textbf{Case (C2).} In all of the following subcases, $(*)$ will hold in column $c$ since the column has both labels $r$ and $r+1$, and $r$ appears above $r+1$ by Lemma \ref{lem:labeling-restriction}.

\begin{itemize}
\item \textbf{Subcase (D0).} We could violate $(*)$ in column $c+1$ is if the cell in column $c$ with label $r$ is label paired with a cell, say $x$, in column $c+1$, but the cell in column $c$ with label $r+1$ is not label paired with a cell in column $c+1$. In this case $\mathcal{L}(x) \ge r+2$, since we must have $\mathcal{L}(x) \ge r$ and we assumed that the labels $r$ and $r+1$ do not appear in column $c+1$. But the cell labeled $r+1$ in column $c$ lies below the cell labeled $r$ by Lemma \ref{lem:labeling-restriction}, showing that $x$ would have preferred to pair with the cell labeled $r+1$. Therefore, there must have been some other cell below $x$ that was already paired with the cell labeled $r+1$ in column $c$ --- otherwise $x$ would pair to the cell labeled $r+1$, instead --- and it follows that $(*)$ is satisfied in column $c+1$ after we label rectify.

\item \textbf{Subcase (D1).} The only way we could run into trouble here is if the cell labeled $r+1$ in column $c+1$ pairs with the cell labeled $r$ in column $c$. However, in this situation, the cell labeled $r+1$ in column $c$ must have already paired with a cell in column $c+1$. Again, $(*)$ holds  in column $c+1$ after label rectifying.

\item \textbf{Subcase (D2).} Here, the cells labeled $r$ and $r+1$ --- denote them as $y_0$ and $y_1$, respectively --- in column $c$ must be label paired to cells in column $c+1$, since both labels appear in column $c+1$ weakly above the corresponding label in column $c$.

It suffices to show that the cell in column $c+1$ with which $y_0$ pairs is strictly above the cell with which $y_1$ pairs. Thus assume for a contradiction that a cell $x$ in column $c+1$ pairs with $y_0$, while $y_1$ is yet unpaired. Then since $y_0$ lies above $y_1$ in column $c$ by Lemma \ref{lem:labeling-restriction}, we must have $\mathcal{L}(y) = r$. However, by the inductive hypothesis, $x$ must lie above the cell in column $c+1$ labeled $r+1$, meaning $y_1$ must have in fact been paired already, a contradiction. Thus $y_0$ always pairs with a cell strictly above the cell with which $y_1$ pairs, and it follows that $(*)$ is satisfied in column $c+1$.
\end{itemize}

Thus property (*) holds throughout rectification and so, too, at the end.
\end{proof}

We use results about the generalized labeling algorithm to complete our proof.

\begin{theorem}
 Let $D$ be a northwest diagram, and let $r$ be a row index such that for every cell in row $r$ there is a cell directly below it row $r+1$. Then for any $T\in\KD(D)$ we have $f_r(T)\in\KD(D)$. In particular $\Dem_r(\KD(D)) = \KD(D)$.
 \label{thm:close_down}
\end{theorem}

\begin{proof}
  Let $U=f_r(T)$ and assume $U\ne T$ or we are done. By \cite[Thm 5.2.2]{Ass-KC} we have $\Label_D(T)$ is well-defined and flagged. By \cite[Lemma 5.3.3]{Ass-KC}, we know $\Label_D(U)$ is well-defined as well. By \cite[Thm 5.3.2]{Ass-KC} $\Rect(\mathcal{L}_D(T))=\mathcal{L}_{\comp{a}}(\Rect(T))$ and $\Rect(\mathcal{L}_D(U))=\mathcal{L}_{\comp{b}}(\Rect(U))$ for some compositions $\comp{a}$ and $\comp{b}$, with the former a Kohnert tableau with respect to $\comp{a}$ since $\mathcal{L}_D(T)$ is flagged. This is to say that $\Rect(T)\in\mathbb{D}(\comp{a})$. Using \cite[Lemma 5.3.3]{Ass-KC} once more we also note that $\comp{a}=\comp{b}$.
  
  From Lemma \ref{lem:row_containment} we know $a_{r+1}\ge a_r$ and from the left-justified case shown in Lemma \ref{lem:left-justified} we have $f_r(\Rect(T))\in\KD(\comp{a})$. Since crystals moves commute with rectification \cite[Thm 4.2.5]{Ass-KC} we write $\Rect(U)\in\KD(\comp{a})$ which implies $\Label_{\comp{a}}(\Rect(U))=\Rect(\Label_D(U))$ is flagged. Rectification does not affect the flagged condition \cite[Lemma 5.3.1]{Ass-KC} so we finally have that $\Label_D(U)$ is flagged and therefore $U\in\KD(D)$ by \cite[Thm 5.2.4]{Ass-KC}.
\end{proof}

\subsection{Recurrence for Kohnert polynomials}
\label{sec:labeling-main}

We can now establish the third and final term in Magyar's recurrence, giving our main result.

\begin{theorem}\label{thm:kohnert_recur}
  The Kohnert polynomials of northwest diagrams satisfy the following recurrence:
    \begin{itemize}
      \item[(K1)] $\kohnert_{\varnothing} = 1$;
      \item[(K2)] if the first column of $D$ is exactly $C_k$, then 
      \begin{equation}\kohnert_{D} = x_1 x_2 \cdots x_k \kohnert_{D-C_k};\end{equation}
      \item[(K3)] if every cell in row $r$ has a cell in row $r+1$ in the same column, then \begin{equation}\label{e:K3}\kohnert_{D} = \pi_r \left(\kohnert_{s_r D}\right).\end{equation}
  \end{itemize}
  In particular, $\kohnert_D = \ch(\fSchur_D)$ for $D$ any northwest diagram.
\end{theorem}

\begin{proof}
  We have observed (K1) by definition, and (K2) is proved in Lemma~\ref{lem:tower}. Assume, then that $D$ is northwest and every cell in row $r$ of $D$, there is a cell in row $r+1$ in the same column. By Theorem~\ref{thm:close_down}, we have $\Dem_r(\KD(D)) = \KD(D)$. Thus applying the Demazure operator to the nested sequence in Corollary~\ref{cor:nested} gives
  \begin{displaymath}
    \Dem_r(\KD(s_r D)) \ \subseteq \ \Dem_r(\KD(D)) \ \subseteq \ \Dem_r(\Dem_r(\KD(s_r D))) \ = \ \Dem_r(\KD(s_r D)),
  \end{displaymath}
  and so $\Dem_r(\KD(s_r D)) = \KD(D)$. By Proposition~\ref{prop:Di}, we have
  \[ \kohnert_D = \ch(\KD(D)) = \ch(\Dem_r(\KD(s_r D))) = \pi_r(\ch(\KD(s_r D))) = \pi_r(\kohnert_{s_r D}). \]
  Thus (K3) holds. By Theorem~\ref{thm:recur}, Kohnert polynomials and characters of flagged Schur modules satisfy the same recurrence, and so must be equal.
\end{proof}

It is worth remarking that Magyar's recurrence, Theorem~\ref{thm:recur}, holds for a more general class of shapes, namely $\%$-avoiding diagrams. A natural question is whether our result holds for this more general class. The answer is a resounding no.

\begin{definition}\label{def:pa}
  A diagram $D$ is $\%$-avoiding if whenever $(j,k),(i,l)\in D$ with $i<j$ and $k<l$, then either $(i,k)\in D$ or $(j,l)\in D$.
\end{definition}

Notice every northwest shape is, in particular, \%-avoiding. 

\begin{lemma}\label{lem:A2}
	Let $z$ be a cell in row $s$ of a diagram $D$. Let $U$ be the diagram consisting of the rows strictly above $s$ weakly below some $r<s$ plus the cells right of $z$ inclusive in row $s$. If $U$ is northwest and there exists a cell in row $r$ in the same column as $z$, then there exists no $T\in \KD(D)$ with $\wt(T)=\wt(D)+K\alpha_{r,s}$ where $K$ is the number of cells right of $z$ inclusive in row $s$.
\end{lemma}

\begin{proof}
	Assume we can in fact construct such a diagram $T$ by performing Kohnert moves on $D$. Note that no Kohnert moves could have moved a cell to a row strictly above $r$ or from a row strictly below row $s$ since we must preserve the weights of rows strictly above $r$ and also those strictly below $s$. Therefore those two regions do not change between $D$ and $T$. 
	
	Let $c$ denote the column containing $z$. Suppose there is a cell $x$ in row $k$ column $b\le c$ of $U$ but that this position is vacant in $T$. Then $x$ must have been moved by a Kohnert move so the number of cells in column $b$ strictly above $k$ must be strictly greater in $T$ than in $D$. In particular, there must be some cell $y$ in row $k'$ with $r\le k'<k$ column $b$ of $T$ whose position is vacant in $D$. The northwest condition on $x$ and the cell in row $r$ column $c$, both lying in $U$, imply there is a cell in row $r$ column $b$ of $D$, so in fact $r<k'<k$. Furthermore, the presence of $x$ along with the vacancy of row $k'$ column $b$ in $D$ imply that row $k'$ has no cells right of $b$ in $U$, hence in $D$ which has an identical row $k'$. 
	
	In order to preserve the row weight of row $k'$ between $D$ and $T$ we then must have that there is a column $b'<b$ such that row $k'$ column $b'$ has a cell $x'$ in $D$ but not in $T$. If we replace $x,k,$ and $b$ with $x'$, $k'$, and $b'$ respectively and we repeat the argument then it never terminates which would imply our diagram $T$ does not actually exist. 
	
	We know that in order to get from $D$ to $T$, precisely $K$ cells must be removed from row $s$ with none more entering. These must be the rightmost $K$ cells in $s$ which includes $z$. Thus, we begin the above contradictory process by letting $x=z$.
\end{proof}

We prove Theorem~\ref{thm:kohnert_recur} is tight with the following result.

\begin{theorem}
	If $D$ is \%-avoiding but not northwest, then $\ch(\fSchur_D) \neq \kohnert_D$.
\end{theorem}

\begin{proof}
    Pick rows $r$ and $s$ with $r<s$ that exhibit a violation of the northwest condition and are minimally separated. That is, the only violations of the northwest condition in the diagram consisting of rows weakly between $r$ and $s$ occur with cells in the rows $r$ and $s$ themselves. Let $y$ be the rightmost cell in row $s$ for which there is a cell strictly right of its column $c$ in row $r$, but not in column $c$ of row $r$. Let $K$ be the number of cells of $D$ in row $s$ weakly right of $y$ for which there is not a corresponding cell in the same column of row $r$. By Proposition~\ref{prop:A1} we have that $x^{\wt(D)+K\alpha_{r,s}}$ is a monomial in the monomial expansion of $\ch(\fSchur_D)$. It now suffices to show that $x^{\wt(D)+K\alpha_{r,s}}$ is not a monomial in $\kohnert_D$.
	
	Let $z$ be $K$th cell from the right in row $s$. Due to the \%-avoiding condition and the fact that there is not a cell in the same column as $y$ of row $r$, we know that for any cell in row $r$ right of $y$, of which there is at least one by assumption, there is a corresponding cell in the same column of row $s$. This corresponding cell in row $s$ is not one of those counted in the definition of $K$. Therefore the number of cells weakly right of $y$ in row $s$ is strictly greater than $K$ which is to say that $z$ lies strictly to the right of $y$.
	
	Now let $U$ be the diagram of $D$ consisting only of the rows strictly above $s$ weakly below $r$, and additionally the cells weakly right of $z$ in row $s$. By choice of $r,s$ and $y$ we know that $U$ is northwest. Moreover, if there were no cell in row $r$ of the same column as $z$, then since $U$ is northwest there would be no cells right of this column in row $r$ in either $U$ or $D$. However, this would mean that the $K$ cells weakly right of $z$ are precisely those cells weakly right of $y$ that do not have a corresponding cell in the same column of row $r$, which is impossible because $y$ is included in this count but lies left of $z$. So instead there must be a cell in row $r$ in the same column as $z$. By Lemma \ref{lem:A2} there is then no $T\in\KD(D)$ with $\wt(T)=\wt(D)+K\alpha_{r,s}$ and therefore $x^{\wt(D)+K\alpha_{r,s}}$ does not appear in the monomial expansion of $\kohnert_D$.
\end{proof}

%
\section*{Acknowledgments}
%

The authors are grateful to Peter Kagey for sharing insights on this project.

%
%

\bibliographystyle{amsalpha}
\bibliography{kohnertmodules}

\end{document}